%% file: HaWaYo1_20231029_arxiv.tex
\documentclass[dvipdfmx-if-dvi,autodetect-engine,ja=standard,reqno]{amsart}
\usepackage{amsmath,amsthm,amssymb}

\usepackage{amsaddr}

\usepackage{mathtools}
\mathtoolsset{showonlyrefs=true} 

\numberwithin{equation}{section}

\usepackage{color}

\usepackage{amsfonts}
\usepackage{mathrsfs}
\usepackage{bbm}

\usepackage[setpagesize=false,colorlinks=true,linkcolor=blue,citecolor=blue]{hyperref}

\newtheorem{definition}{Definition}[section]
\newtheorem{proposition}[definition]{Proposition}
\newtheorem{theorem}[definition]{Theorem}
\newtheorem{lemma}[definition]{Lemma}

\newtheorem{remark}[definition]{Remark}

\newcommand{\al}{\alpha}
\newcommand{\bet}{\beta}

\pagebreak

\title[Nonlinear damped beam equation]{Asymptotic profiles for the Cauchy problem of damped beam equation with two
variable coefficients and derivative nonlinearity
}
\author[M. A. Hamza, Y. Wakasugi and S. Yoshikawa]{Mohamed Ali Hamza${}^1$, Yuta Wakasugi${}^{2*}$ and Shuji Yoshikawa${}^3$}
\email{mahamza@iau.edu.sa}
\email{wakasugi@hiroshima-u.ac.jp}
\email{yoshikawa@oita-u.ac.jp}

\address{${}^1$
Basic Sciences Department,
Deanship of Preparatory Year and Supporting Studies,
P. O. Box 1982, Imam Abdulrahman Bin Faisal University,
Dammam, KSA.%
}

\address{${}^2$
Laboratory of Mathematics,
Graduate School of Advanced Science and Engineering,
Hiroshima University,
Higashi-Hiroshima, 739-8527, Japan%
}

\address{${}^3$
Division of Mathematical Sciences,
Faculty of Science and Technology,
Oita University,
Oita, 870-1192, Japan%
}

\date{\today}


\begin{document}

\maketitle

\begin{abstract}
In this article we investigate the asymptotic profile of solutions for the Cauchy problem of the nonlinear damped beam equation with two variable coefficients: 
\[
\partial_t^2 u +  b(t) \partial_t u - a(t) \partial_x^2 u + \partial_x^4 u
    = \partial_x \left( N(\partial_x u) \right). 
\] 
In the authors' previous article \cite{YoWa21}, the asymptotic profile of solutions for linearized problem ($N \equiv 0$) was  
classified depending on the assumptions for the coefficients $a(t)$ and $b(t)$ and proved the asymptotic behavior in effective damping cases. 
We here give the conditions of the coefficients and the nonlinear term in order that the solution behaves as the solution for the heat equation: $b(t) \partial_t u - a(t) \partial_x^2 u=0$ asymptotically as $t \to \infty$.   
\end{abstract}

\footnote[0]{$*$ Corresponding author}

\section{Introduction}
We study the Cauchy problem of nonlinear
damped beam equation
\begin{align}
\label{eq:ndb}
    \left\{
    \begin{alignedat}{3}
    &\partial_t^2 u + b(t) \partial_t u - a(t) \partial_x^2 u + \partial_x^4 u
    = \partial_x \left( N(\partial_x u) \right),
    &\qquad&
    t \in (0,\infty), x \in \mathbb{R},\\
    &u(0,x) = u_0(x), \ \partial_t u(0,x) = u_1(x),
    &\qquad&
    x \in \mathbb{R},
    \end{alignedat}
    \right.
\end{align}
where
$u= u(t,x)$
is a real-valued unknown,
$a(t)$ and $b(t)$ 
are given positive functions of $t$,
$N(\partial_x u)$
denotes the nonlinear function,
and
$u_0$ and $u_1$
are given initial data.

Before giving more precise assumptions for $a(t)$, $b(t)$ and $N$ and our result, 
we first mention the physical background and the mathematical motivations of the problem. 
The equation \eqref{eq:ndb} corresponds to the so-called Falk model under isothermal assumption with the damping 
term.   
The Falk model is one of the models for a thermoelastic deformation with austenite-martensite phase transitions on shape memory alloys: 
\begin{equation*}
\begin{cases}
\partial_t^2 u  + \partial_x^4 u
    = \partial_x \left\{ (\theta  - \theta_c) \partial_x u  - (\partial_x u)^3 + (\partial_x u)^5 \right\},  \qquad & 
    \\
\partial_t \theta - \partial_x^2 \theta = \theta \partial_x u \partial_t \partial_x u,  \qquad &t \in (0,\infty), x \in \mathbb{R}, \\
u(0, x) = u_0(x), \ \partial_t u(0,x) = u_1(x), \ \theta(0,x) = \theta_0(x),  \qquad &x \in \mathbb{R},
\end{cases}
\end{equation*}
where $u$ and $\theta$ are the displacement and the absolute temperature, respectively, and $\theta_c$ is a positive constant representing the critical temperature for the phase transition. 
If we assume the temperature are controllable and uniformly distributed with respect to the space, that is, $\theta$ is given function uniform in $x$ such as $\theta - \theta_c = a(t)$ and set 
$N(\varepsilon) = \varepsilon^5 - \varepsilon^3$, then the problem 
\eqref{eq:ndb} is surely derived. 
Our interest directs to the behavior of solution around the initial temperature $\theta_0$ is closed to the critical temperature $\theta_c$. 
Indeed, the Lyapunov stability for the solution of the Falk model is shown in \cite{su-yo}, which claims that the temperature tends to the function uniformly distributed in $x$ in the bounded domain case. 
For more precise information of the Falk model of shape memory alloys, 
we refer the reader to Chapter 5 in \cite{br-sp}.  
We are also motivated by the extensible beam equation proposed by Woinovsky-Krieger \cite{wo}: 
\begin{equation}\label{wo-kr}
 \partial_t^2 u - \left( \int_{\mathbb{R}} | \partial_x u|^2 dx   \right) \partial_x^2 u + \partial_x^4 u
    = 0. 
\end{equation}
In \cite{ba2}, the model with the damping term was proposed and the stability result was shown. 
The problem \eqref{eq:ndb} corresponds to the nonlinear generalization for the equation. 
As the observation similar to the Kirchhoff equation, the linearized 
problem substituting the given function $a(t)$ into the nonlocal term was also studied by e.g.
\cite{DaEb}, \cite{ra-yo},  \cite{YoWa21} and \cite{Li-MuAr}. 

Next, let us explain the mathematical background of our problem. 
It is well-known that the solution of the Cauchy problem for the damped wave equation $\partial_t^2 u + \partial_t u - \partial_x^2 u=0$ behaves as the solution for the heat equation $\partial_t u - \partial_x^2 u =0$ asymptotically as $t \to \infty$ (see e.g. \cite{MaNi03}). 
Roughly speaking, this implies that $\partial_t^2 u$ decays faster than $\partial_t u$ as $t \to \infty$. 
From the same observation, the solution for the beam equation $\partial_t^2 u + \partial_t u - \partial_x^2 u +\partial_x^4 u=0$ 
behaves as the solution for the heat equation $\partial_t u - \partial_x^2 u =0$ asymptotically as $t \to \infty$, 
because $\partial_x^4 u$ decays faster than $\partial_x^2 u$. 
The above observation induces the investigation of the solution for the equation with time variable coefficient: $\partial_t^2 u + b(t) \partial_t u - \partial_x^2 u=0$ with $b(t) \sim (1+t)^{\beta}$. 
The precise analysis implies that the solution behaves as the solution for the heat equation $b(t) \partial_t u - \partial_x^2 u =0$ when $\beta < -1$, and on the other hand, that the solution behaves as the solution for the wave equation $\partial_t^2 u - \partial_x^2 u =0$ when $-1 < \beta <1$ (see e.g. \cite{Mo76}, \cite{Ni2}, \cite{Wi06} and \cite{Wi07JDE}). 
Correspondingly, the authors in \cite{YoWa21} studied the asymptotic behavior of the solution for  the linearized problem of \eqref{eq:ndb} 
\[
\partial_t^2 u + b(t) \partial_t u - a(t) \partial_x^2 u + \partial_x^4 u
    = 0. 
\] 
As in Figure 1, 
we divide the two-dimensional regions $\Omega_j$ for $(\al, \bet)$ ($j=1,2,3,4,5$) by 
\begin{align*}
\Omega_1 &:= \left\{ (\al,\bet) \in \mathbb{R}^2 \mid  -1 < \bet <  \min\{ \al + 1,  2\al +1 \} \right\}, \\
\Omega_2 &:= \left\{ (\al,\bet) \in \mathbb{R}^2 \mid  \max\{ -1,  2\al +1 \} < \bet <  1 \right\}, \\
\Omega_3 &:= \left\{ (\al,\bet) \in \mathbb{R}^2 \mid   \bet <  -1 < \alpha \right\}, \\
\Omega_4 &:= \left\{ (\al,\bet) \in \mathbb{R}^2 \mid  \bet <  -1, \alpha <-1 \right\}, \\
\Omega_5 &:= \left\{ (\al,\bet) \in \mathbb{R}^2 \mid  \max\{ 1,  \al +1 \} < \bet  \right\}.
\end{align*}

\hspace*{1cm} 
\input{zu2.tex}
\vspace*{3mm}

By a scaling argument, the result gives the conjecture for 
the classification of the asymptotic behavior of the solution: 
\begin{enumerate}
\item In $(\al, \bet) \in \Omega_1$, $u(t)$ behaves as the solution for $b(t) u_t - a(t) u_{xx} = 0$. 
\item In $(\al, \bet) \in \Omega_2$, $u(t)$ behaves as the solution for $b(t) u_t + u_{xxxx} = 0$. 
\item In $(\al, \bet) \in \Omega_3$, $u(t)$ behaves as the solution for $u_{tt} - a(t) u_{xx} = 0$. 
\item In $(\al, \bet) \in \Omega_4$, $u(t)$ behaves as the solution for $u_{tt} + u_{xxxx} = 0$. 
\item In $(\al, \bet) \in \Omega_5$, $u(t)$ behaves as the solution for over damping case. 
\end{enumerate}
As a partial answer, the authors proved the effective damping cases (1) and (2) in \cite{YoWa21}.  
Here we shall give the result for the nonlinear problem \eqref{eq:ndb} in the case (1).  

From now on, we shall give our main result. 
To state it precisely, we put the following assumptions.

\noindent
\textbf{Assumption (A)}
The coefficients
$a(t)$ and $b(t)$
are smooth positive functions satisfying
\begin{align}
    C^{-1} (1+t)^{\alpha} \le a(t) \le C (1+t)^{\alpha},\quad
    C^{-1} (1+t)^{\beta} \le b(t) \le C (1+t)^{\beta}
\end{align}
and
\begin{align}
	|a'(t)| \le C (1+t)^{\alpha-1},\quad
	|b'(t)| \le C (1+t)^{\beta-1}
\end{align}
with some
constant $C \ge 1$
and the parameters
$\alpha, \beta \in \mathbb{R}$.
Moreover, we assume
\begin{align}
    (\alpha, \beta) \in \Omega_1
    := \left\{ (\alpha, \beta) \in \mathbb{R}^2 \mid
    -1 < \beta < \min \{ \alpha + 1, 2\alpha + 1 \} \right\}.
\end{align}

\noindent
\textbf{Assumption (N)}
The function
$N(\partial_x u)$
is a linear combination of
$(\partial_x u)^2$
and $p$-th order terms with $p \ge 3$.
More precisely,
the function
$N$
has the form
\begin{align}
    N(z) = \mu z^2 + \tilde{N}(z),
\end{align}
with some $\mu \in \mathbb{R}$,
where
$\tilde{N} \in C^2(\mathbb{R})$
satisfies
\begin{align}
    \tilde{N}^{(j)}(0)= 0
    \quad
    \text{and}
    \quad
    | \tilde{N}^{(j)}(z) - \tilde{N}^{(j)}(w) |
    \le C (|z| + |w|)^{p-1-j} |z-w|
    \quad (z,w \in \mathbb{R})
\end{align}
holds with some $p\ge 3$ for $j = 0,1,2$.

A typical example of our nonlinearity is
\begin{align}
    \partial_x \left( N(\partial_x u) \right)
    = \partial_x ( \partial_x u)^2
    + \partial_x \left( |\partial_x u|^{p-1} \partial_x u \right)
\end{align}
with $p \ge 3$.

\begin{remark}\label{rem:11}
The assumption (A) implies
\begin{align}
	\frac{-\beta + 1}{\alpha - \beta + 1} < 2 < p.
\end{align}
This means that the nonlinearity is supercritical
(see the argument in Section 4.1).
\end{remark}

We further prepare the following notations.
Let $G = G(t,x)$ be the Gaussian, that is,
\begin{align}
	G(t,x) = \frac{1}{\sqrt{4\pi t}} \exp \left( - \frac{x^{2}}{4t} \right).
\end{align}
We define
\begin{align}
	r(t) = \frac{a(t)}{b(t)} \quad
	\text{and} \quad
	R(t) = \int_{0}^{t} r(\tau) \,d\tau.
\end{align}
Remark that
\begin{align}
	C^{-1} (1+t)^{\alpha-\beta+1} \le R(t) \le C (1+t)^{\alpha-\beta+1}
\end{align}
holds with some constant $C \ge 1$, thanks to the assumption (A). 

\begin{theorem}
Under the assumptions (A) and (N),
there exists a constant
$\varepsilon_{0} > 0$
such that if
$(u_{0}, u_{1}) \in \left( H^{2,1}(\mathbb{R}) \cap H^{3,0}(\mathbb{R}) \right)
\times \left( H^{0,1}(\mathbb{R})\cap H^{1,0}(\mathbb{R}) \right)$
and
\begin{align}
	\| u_{0} \|_{H^{2,1}\cap H^{3,0}}
	+ \| u_{1} \|_{H^{0,1}\cap H^{1,0}}
	\le \varepsilon_{0},
\end{align}
then there exists a unique solution
\begin{align}\label{eq:sol:class}
	u \in C([0,\infty); H^{2,1}(\mathbb{R}) \cap H^{3,0}(\mathbb{R}))
	\cap C^{1}([0,\infty); H^{0,1}(\mathbb{R})\cap H^{1,0}(\mathbb{R})).
\end{align}
Moreover, the solution
$u$ has the asymptotic behavior
\begin{align}
	\| u(t,\cdot) - m^{*} G(R(t), \cdot ) \|_{L^{2}}
	\le
	C (R(t)+1)^{-\frac{1}{4}-\frac{\lambda}{2}}
	\left( \| u_{0} \|_{H^{2,1}\cap H^{3,0}} + \| u_{1} \|_{H^{0,1}\cap H^{1,0}} \right)
\end{align}
with some constants
$C>0$, $m^{*} \in \mathbb{R}$,
and $\lambda > 0$.
\end{theorem}

\begin{remark}
From the proof, $\lambda$ is taken to be arbitrary so that
\[
	0 < \lambda < \min \left\{ \frac{1}{2}, \frac{2(\beta+1)}{\alpha-\beta+1}, \frac{2\alpha -\beta +1}{\alpha - \beta +1} \right\}.
\]
\end{remark}

The proof is based on the method by Gallay and Raugel \cite{GaRa98} using the self-similar 
transformation and the standard energy method. 

This paper is organized as follows.
In Section 2, we rewrite the problem through the self-similar transformation. 
Thereafter, we show several energy estimates in Section 3, and give a priori estimates through 
the estimates for the nonlinear terms in Section 4.  
In the appendix, we also give a lemma for energy identities which is frequently used in the proof 
and the proof of local-in-time existence of solution for the readers' convenience. 

At the end of this section we prepare notation and several definitions used throughout this paper. 
We denote by $C$ a positive constant, which may change from line to line.
The symbol
$a(t) \sim b(t)$
means that
$C^{-1} b(t) \le a(t) \le  C b(t)$
holds for some constant
$C \ge 1$.
$L^p = L^p(\mathbb{R})$
stands for the usual Lebesgue space,
and
$H^{k,m} = H^{k,m}(\mathbb{R})$
for $k \in \mathbb{Z}_{\ge 0}$
and $m \in \mathbb{R}$
is the weighted Sobolev space defined by
\begin{align*}
	H^{k,m}(\mathbb{R}) = \left\{ f \in L^2(\mathbb{R}) ;
			\| f \|_{H^{k,m}} = \sum_{\ell = 0}^k \| (1+|x|)^m \partial_x^{\ell} f \|_{L^2} < \infty \right\}.
\end{align*}

\section{Scaling variables}
The local existence of the solution in the class \eqref{eq:sol:class}
is standard (see Appendix B).
Thus, it suffices to show the a priori estimate.
To prove it,
following the argument of Gallay and Raugel \cite{GaRa98},
we introduce the scaling variables
\begin{align}\label{eq:scaling:var}
	s = \log (R(t)+1),\quad y=\frac{x}{\sqrt{R(t)+1}},
\end{align}
and define
$v = v(s,y)$ and $w = w(s,y)$ by
\begin{align}\label{eq:def:v}
	u(t,x) &= \frac{1}{\sqrt{R(t)+1}} v \left( \log(R(t)+1), \frac{x}{\sqrt{R(t)+1}} \right), \\
	 u_{t} (t,x) &= \frac{R'(t)}{(R(t)+1)^{3/2}} w \left( \log(R(t)+1), \frac{x}{\sqrt{R(t)+1}} \right).
\end{align}
Then, by a straightforward computation, the Cauchy problem \eqref{eq:ndb} is rewritten as
\begin{align}\label{eq:v:w}
	\left\{ \begin{alignedat}{3}
	&v_{s} - \frac{y}{2} v_{y} - \frac{1}{2} v = w,\\
	&\frac{r^{2}e^{-s}}{a} \left( w_{s} - \frac{y}{2} w_{y} - \frac{3}{2} w \right) + \left( 1+ \frac{r'}{a} \right) w
	= v_{yy} - \frac{e^{-s}}{a} v_{yyyy}
	+ \frac{e^{s}}{a} \partial_{y} \left( N \left( e^{-s} v_{y} \right) \right), \\
	&v(0,y) = v_{0}(y), \ w(0,y) = w_{0}(y),
	\end{alignedat} \right.
\end{align}
where
$v_{0}(y) = u_{0}(y)$
and $w_{0}(y) = \frac{1}{r(0)} u_{1}(y)$.
Here, we also note that the functions
$a, b, r, r'$ appearing the above precisely mean such as
$a(t(s)) = a(R^{-1}(e^{s}-1))$.

\begin{remark}
When $N(z) = |z|^{p-1}z$, the nonlinearity has a bound
\begin{align}
	\left| \frac{e^{s}}{a(t(s))} \partial_{y} \left( N ( e^{-s} v_{y} ) \right) \right|
	\le
	C e^{\left(\frac{-\beta+1}{\alpha-\beta+1} - p\right)s}
	|v_{y}|^{p-1} |v_{yy}|.
\end{align}
Since
$\frac{-\beta+1}{\alpha - \beta +1} < 2$,
the assumption (N) implies that the nonlinearity
can be treated as remainder.
\end{remark}

To investigate the asymptotic behavior of the solution of \eqref{eq:v:w},
we define
\begin{align}\label{eq:def:m}
	m(s) := \int_{\mathbb{R}} v(s,y) \,dy.
\end{align}
By the first equation of \eqref{eq:v:w} and the integration by parts, we have
\begin{align}
	m_{s}(s) = \frac{d}{ds}m(s) = \int_{\mathbb{R}} v_{s} \,dy
	= \int_{\mathbb{R}} \left( \frac{y}{2}v_{y} + \frac{1}{2}v + w \right)\,dy
	= \int_{\mathbb{R}} w \,dy.
\end{align}
We also define
\begin{align}
\label{eq:def:phi}
	\phi(y) &:= G(1,y) = \frac{1}{\sqrt{4\pi}} \exp \left( - \frac{y^{2}}{4} \right),\\
\label{eq:def:psi}
	\psi(y) &:= \phi_{yy}(y).
\end{align}
Using them, we decompose $(v,w)$ as
\begin{align}\label{eq:def:fg}
	\begin{aligned}
	v(s,y) &= m(s) \phi(y) + f(s,y),\\
	w(s,y) &= m_{s}(s) \phi(y) + m(s) \psi(y) + g(s,y),
	\end{aligned}
\end{align}
and we expect that the functions
$f$ and $g$ defined above can be treated as remainders.

By a direct calculation, we obtain the following equation for $m(s)$:

\begin{lemma}\label{lem:eq:m}
We have
\begin{align}
    \frac{r^2 e^{-s}}{a} \left( m_{ss} - m_s \right)
    = - \left( 1 + \frac{r'}{a} \right) m_s.
\end{align}
\end{lemma}

From the above lemma and the straightforward computation,
we can see that
$(f,g)$
satisfies the following equations.
\begin{align}\label{eq:f:g}
	\left\{ \begin{alignedat}{3}
	&f_{s} - \frac{y}{2} f_{y} - \frac{1}{2} f = g,\\
	&\frac{r^{2}e^{-s}}{a} \left( g_{s} - \frac{y}{2} g_{y} - \frac{3}{2} g \right) + \left( 1+ \frac{r'}{a} \right) g
	= f_{yy} - \frac{e^{-s}}{a} f_{yyyy}
	+ \frac{e^{s}}{a} \partial_{y} \left( N \left( e^{-s} v_{y} \right) \right) + h,
	\end{alignedat} \right.
\end{align}
where
\begin{align}\label{eq:h}
    h &=
    - \frac{r^2 e^{-s}}{a} \left( 2 m_s \psi - \frac{y}{2}m \psi_y - \frac{3}{2} m \psi \right)
    - \frac{r'}{a} m \psi - \frac{e^{-s}}{a} m \psi_{yy}.
\end{align}
They satisfy
\begin{align}\label{eq:fgh:0}
	\int_{\mathbb{R}} f(s,y)\,dy
	= \int_{\mathbb{R}} g(s,y) \,dy
	= \int_{\mathbb{R}} h(s,y) \,dy
	= 0.
\end{align}
Therefore, our goal is
to give energy estimates of
the solutions
$(f,g)$
to the equation
\eqref{eq:f:g}
under the condition \eqref{eq:fgh:0}.

\section{Energy estimates}
In this section, we give
energy estimates of
$(f,g)$
defined by \eqref{eq:def:fg}.
We first prepare the following
general lemma for energy identities.
\begin{lemma}\label{lem:YoWa:en}
Let
$l,m\in \mathbb{R}$,
$n \in \mathbb{N}\cup \{0\}$,
and let
$c_1(s), c_2(s), c_4(s)$
be smooth functions defined on
$[0,\infty)$.
We consider a system for two functions
$f=f(s,y)$ and $g=g(s,y)$
given by
\begin{align}\label{eq:lem:YoWa}
    \left\{ \begin{alignedat}{3}
    &f_{s}- \frac{y}{2} f_{y}-l f = g,\\
    &c_{1}(s) \left( g_{s}- \frac{y}{2} g_{y} - m g \right)
    + c_{2}(s) g + g
    =  f_{yy} - c_{4}(s) f_{yyyy}+h
    \end{alignedat} \right.
    \quad
    (s,y) \in (0,\infty) \times \mathbb{R},
\end{align}
where
$h = h(s,y)$
is a given smooth function belonging to
$C([0,\infty); H^{0,n}(\mathbb{R}))$.
We define the energies
\begin{align}
    E_1(s)
    &=
    \frac{1}{2}\int_{\mathbb{R}}
    y^{2n}
    \left(
    f_y^2 + c_4(s) f_{yy}^2 + c_1(s)g^2
    \right) \,dy,\\
    E_{2}(s)
    &=
    \int_{\mathbb{R}}
    y^{2n}
    \left(
    \frac{1}{2} f^{2}
    + c_{1}(s) f g
    \right) \,dy.
\end{align}
Then, we have
\begin{align}
    \frac{d}{ds} E_1(s)
    &=
    - \int_{\mathbb{R}} y^{2n} g^2 \,dy 
    + \left( - \frac{2n-1}{4} + l \right)
    \int_{\mathbb{R}} y^{2n} f_y^2 \,dy
    + \left( - \frac{2n-3}{4} + l \right)c_4(s)
    \int_{\mathbb{R}} y^{2n} f_{yy}^2 \,dy \\
    &\quad +
    \left( - \frac{2n+1}{4}+ m \right) c_1(s)
    \int_{\mathbb{R}} y^{2n} g^2 \,dy
    - c_2(s) \int_{\mathbb{R}} y^{2n} g^2 \,dy \\
    &\quad
    -2n \int_{\mathbb{R}} y^{2n-1} f_y g \,dy
    - 2n(2n-1) c_4(s) \int_{\mathbb{R}} y^{2n-2} f_{yy} g \,dy
    - 4n c_4(s) \int_{\mathbb{R}} y^{2n-1} f_{yy} g_y \,dy \\
    &\quad
    + \frac{c_4'(s)}{2} \int_{\mathbb{R}} y^{2n} f_{yy}^2 \,dy
    + \frac{c_1'(s)}{2} \int_{\mathbb{R}} y^{2n} g^2 \,dy
    + \int_{\mathbb{R}} y^{2n} gh \,dy
\end{align}
and
\begin{align}
    \frac{d}{ds} E_2(s)
    &=
    - \int_{\mathbb{R}} y^{2n} f_y^2 \,dy
    -c_4(s) \int_{\mathbb{R}} y^{2n} f_{yy}^2 \,dy 
    + \left( - \frac{2n+1}{4} + l \right)
    \int_{\mathbb{R}} y^{2n} f^2 \,dy \\
    &\quad
    + c_1(s) \int_{\mathbb{R}} y^{2n} g^2 \,dy
    + \left( - \frac{2n+1}{2}  + l + m \right) c_1(s)
    \int_{\mathbb{R}} y^{2n} f g \,dy
    - c_2(s) \int_{\mathbb{R}} y^{2n} f g \,dy  \\
    &\quad
    -2n \int_{\mathbb{R}} y^{2n-1} f f_y \,dy 
    - 4n c_4(s) \int_{\mathbb{R}} y^{2n-1} f_y f_{yy} \,dy
    - 2n(2n-1) c_4(s) \int_{\mathbb{R}} y^{2n-2} f f_{yy} \,dy \\
    &\quad
    + c_1'(s) \int_{\mathbb{R}} y^{2n} f g \,dy
    + \int_{\mathbb{R}} y^{2n} f h \,dy.
\end{align}
\end{lemma}
The case
$n = 0$
is given by
\cite[Lemma 3.1]{YoWa21}.
We will prove a slightly more general version of this lemma in Appendix A.

To bring out the decay property
of the solutions $(f,g)$ to \eqref{eq:f:g}
from the condition \eqref{eq:fgh:0},
we define the auxiliary functions
\begin{align}
    F(s,y) := \int_{-\infty}^y f(s,z) \,dz,\quad
    G(s,y) := \int_{-\infty}^y g(s,z) \,dz,\quad
    H(s,y) := \int_{-\infty}^y h(s,z) \,dz.
\end{align}
Then, by the following Hardy inequality,
the conditions
\eqref{eq:fgh:0} and
$f(s), g(s) \in H^{0,1}(\mathbb{R})$
ensure
$F(s), G(s) \in L^2(\mathbb{R})$.

\begin{lemma}[Hardy-type inequality {\cite[Lemma 3.9]{Wa17}}] \label{lem:hardy}
Let
$f = f(y) \in H^{0,1}(\mathbb{R})$
and satisfy
$\int_{\mathbb{R}} f(y) \,dy = 0$.
Let
$F(y) = \int_{-\infty}^y f(z) \,dz$.
Then, we have
\begin{align}
    \int_{\mathbb{R}} F(y)^2 \,dy
    \le 4 \int_{\mathbb{R}} y^2 f(y)^2 \,dy.
\end{align}
\end{lemma}

From \eqref{eq:f:g},
$F$ and $G$ satisfy the following system.
\begin{align}\label{eq:F:G}
	\left\{ \begin{alignedat}{3}
	&F_{s} - \frac{y}{2} F_{y} = G,\\
	&\frac{r^{2}e^{-s}}{a} \left( G_{s} - \frac{y}{2} G_{y} - G \right) + \left( 1+ \frac{r'}{a} \right) G
	= F_{yy} - \frac{e^{-s}}{a} F_{yyyy}
	+ \frac{e^{s}}{a}  N \left( e^{-s} v_{y} \right) + H.
	\end{alignedat} \right.
\end{align}
We define the energies of
$(F,G)$
by
\begin{align}
    E_{01}(s)
    &:=
    \frac{1}{2} \int_{\mathbb{R}} F_y(s,y)^2 \,dy
    + \frac{e^{-s}}{2a}
    \int_{\mathbb{R}} F_{yy}(s,y)^2 \,dy 
    + \frac{r^2e^{-s}}{2a}
    \int_{\mathbb{R}} G(s,y)^2 \,dy,\\
    E_{02}(s)
    &:=
    \frac{1}{2} \int_{\mathbb{R}} F(s,y)^2 \,dy
    + \frac{r^2e^{-s}}{a}
    \int_{\mathbb{R}} F(s,y) G(s,y) \,dy.
\end{align}
\begin{lemma}\label{lem:E0}
We have
\begin{align}
    \frac{d}{ds} E_{01}(s)
    + \int_{\mathbb{R}} G^2 \,dy
    &=
    \frac{1}{2} E_{01}(s)
    - \frac{a'}{2ra^2} \int_{\mathbb{R}} F_{yy}^2 \,dy
    - \frac{ra'}{2a^2} \int_{\mathbb{R}} G^2 \,dy \\
    &\quad
    + \int_{\mathbb{R}} G \frac{e^{s}}{a}
    N \left( e^{-s} v_{y} \right) \,dy
    + \int_{\mathbb{R}} GH \,dy,\\
    \frac{d}{ds} E_{02}(s)
    + \frac{1}{2} E_{02}(s)
    + 2 E_{01}(s)
    &=
    2 \frac{r^2e^{-s}}{a} \int_{\mathbb{R}} G^2 \,dy
     + \left( \frac{r'}{a}- \frac{ra'}{a^2} \right)
     \int_{\mathbb{R}} FG \,dy \\
     &\quad
     + \int_{\mathbb{R}} F \frac{e^s}{a} N\left( e^{-s} v_y \right) \,dy
     + \int_{\mathbb{R}} FH \,dy.
\end{align}
\end{lemma}
\begin{proof}
We apply Lemma \ref{lem:YoWa:en} as $f=F$, $g=G$ and $h=\frac{e^s}{a}N(e^{-s}v_y) +H$  
with
$l = 0$,
$m = 1$,
$n = 0$,
$c_1(s) = r^2e^{-s}/a$,
$c_2(s) = r'/a$,
and
$c_4(s) = e^{-s}/a$.
Noting that
\begin{align}
    \frac{d}{ds} r(t(s))
    = \frac{d}{ds} r(R^{-1}(e^s-1))
    = \frac{r'(t(s))}{r(t(s))} e^s,
\end{align}
we first have
\begin{align}
\label{eq:c1prime}
    c_1'(s)
    &= \frac{1}{a^2}
    \left( 2 r \frac{dr}{ds} a e^{-s}
    -r^2 a e^{-s}
    -r^2 \frac{da}{ds} e^{-s}
    \right)
    =
    \frac{2r'}{a} - \frac{r^2 e^{-s}}{a}
    - \frac{ra'}{a^2},\\
\label{eq:c4prime}
    c_4'(s)
    &= \frac{1}{a^2}
    \left( -a e^{-s} - \frac{da}{ds} e^{-s}
    \right) 
    = -\frac{e^{-s}}{a} - \frac{a'}{ra^2}.
\end{align}
Thus, we obtain
\begin{align}
    \frac{d}{ds} E_{01}(s)
    &=
    - \int_{\mathbb{R}} G^2 \,dy
    + \frac{1}{4} \int_{\mathbb{R}} F_y^2 \,dy
    + \frac{3e^{-s}}{4a} \int_{\mathbb{R}} F_{yy}^2 \,dy \\
    &\quad + \frac{3r^2e^{-s}}{4a}
    \int_{\mathbb{R}} G^2 \,dy
    - \frac{r'}{a} \int_{\mathbb{R}} G^2 \,dy  - \frac{1}{2}
    \left( \frac{e^{-s}}{a} +
    \frac{a'}{ra^2} \right)
    \int_{\mathbb{R}} F_{yy}^2 \,dy \\
    &\quad 
    + \frac{1}{2}
    \left( \frac{2 r'}{a} 
    - \frac{r^2 e^{-s}}{a}
    - \frac{r a'}{a^2}
    \right)
    \int_{\mathbb{R}} G^2 \,dy\\
    &\quad
    + \int_{\mathbb{R}} G \frac{e^{s}}{a}
    N \left( e^{-s} v_{y} \right) \,dy
    + \int_{\mathbb{R}} GH \,dy \\
    &=
    - \int_{\mathbb{R}} G^2 \,dy
    + \frac{1}{2} E_{01}(s)
    - \frac{a'}{2ra^2} \int_{\mathbb{R}} F_{yy}^2 \,dy
    - \frac{ra'}{2a^2} \int_{\mathbb{R}} G^2 \,dy \\
    &\quad
    + \int_{\mathbb{R}} G \frac{e^{s}}{a}
    N \left( e^{-s} v_{y} \right) \,dy
    + \int_{\mathbb{R}} GH \,dy.
\end{align}
Next, we compute the derivative of
$E_{02}(s)$.
By Lemma \ref{lem:YoWa:en},
we obtain
\begin{align}
    \frac{d}{ds} E_{02}(s)
    &=
    - \int_{\mathbb{R}} F_y^2 \,dy
    - \frac{e^{-s}}{a} \int_{\mathbb{R}} F_{yy}^2 \,dy
     - \frac{1}{4} \int_{\mathbb{R}} F^2 \,dy \\
     &\quad
     + \frac{r^2 e^{-s}}{a} \int_{\mathbb{R}} G^2 \,dy
     + \frac{r^2 e^{-s}}{2a} \int_{\mathbb{R}} FG \,dy
     - \frac{r'}{a} \int_{\mathbb{R}} FG \,dy \\
     &\quad
     + \int_{\mathbb{R}} F \frac{e^s}{a} N\left( e^{-s} v_y \right) \,dy
     + \int_{\mathbb{R}} FH \,dy \\
     &\quad
     + \left( \frac{2r'}{a} - \frac{r^2 e^{-s}}{a} - \frac{ra'}{a^2} \right)
     \int_{\mathbb{R}} FG \,dy \\
     &=
     - \frac{1}{2} E_{02}(s) - 2E_{01}(s) \\
     &\quad
     + 2 \frac{r^2e^{-s}}{a} \int_{\mathbb{R}} G^2 \,dy
     + \left( \frac{r'}{a}- \frac{ra'}{a^2} \right)
     \int_{\mathbb{R}} FG \,dy \\
     &\quad
     + \int_{\mathbb{R}} F \frac{e^s}{a} N\left( e^{-s} v_y \right) \,dy
     + \int_{\mathbb{R}} FH \,dy.
\end{align}
This completes the proof.
\end{proof}

Next, for
$n = 0, 1$,
we define the energies of
$(f,g)$
by
\begin{align}
    E_{11}^{(n)}(s)
    &:=
    \frac{1}{2} \int_{\mathbb{R}} y^{2n} f_y(s,y)^2 \,dy
    + \frac{e^{-s}}{2a}
    \int_{\mathbb{R}} y^{2n} f_{yy}(s,y)^2 \,dy 
    + \frac{r^2e^{-s}}{2a}
    \int_{\mathbb{R}} y^{2n} g(s,y)^2 \,dy,\\
    E_{12}^{(n)}(s)
    &:=
    \frac{1}{2} \int_{\mathbb{R}} y^{2n} f(s,y)^2 \,dy
    + \frac{r^2e^{-s}}{a}
    \int_{\mathbb{R}} y^{2n} f(s,y) g(s,y) \,dy.
\end{align}

\begin{lemma}\label{lem:E1}
For
$n= 0, 1$,
we have
\begin{align}
    \frac{d}{ds} E_{11}^{(n)} (s)
    + \int_{\mathbb{R}} y^{2n} g^2 \,dy
    &=
    \frac{3-2n}{2} E_{11}^{(n)}(s)
    -2n \int_{\mathbb{R}} y^{2n-1} f_y g \,dy
    \\
    &\quad
    -2n(2n-1) \frac{e^{-s}}{a} \int_{\mathbb{R}} y^{2n-2} f_{yy}g \,dy
    -4n \frac{e^{-s}}{a} \int_{\mathbb{R}} y^{2n-1} f_{yy} g_{y} \,dy \\
    &\quad
    -\frac{a'}{2ra^2} \int_{\mathbb{R}} y^{2n} f_{yy}^2 \,dy
    - \frac{ra'}{2a^2} \int_{\mathbb{R}} y^{2n} g^2 \,dy \\
    &\quad
    + \int_{\mathbb{R}} y^{2n} g
    \left( \frac{e^{s}}{a} \partial_y (N(e^{-s} v_y)) + h \right) \,dy
\end{align}
and
\begin{align}
    \frac{d}{ds} E_{12}^{(n)} (s)
    + \frac{1}{2} E_{12}^{(n)}(s)
    + 2 E_{11}^{(n)}(s)
    &=
    (1-n)E_{12}^{(n)}(s)
    + \frac{2r^2e^{-s}}{a} \int_{\mathbb{R}} y^{2n} g^2 \,dy
    -2n \int_{\mathbb{R}} y^{2n-1} f f_y \,dy
    \\
    &\quad
    -4n \frac{e^{-s}}{a} \int_{\mathbb{R}} y^{2n-1} f_y f_{yy} \,dy
    -2n (2n-1) \frac{e^{-s}}{a} \int_{\mathbb{R}} y^{2n-2} f f_{yy} \,dy \\
    &\quad
    + \left( \frac{r'}{a} - \frac{ra'}{a^2} \right) \int_{\mathbb{R}} y^{2n} fg \,dy
    + \int_{\mathbb{R}} y^{2n} f
    \left( \frac{e^{s}}{a} \partial_y (N(e^{-s}v_y)) + h \right) \,dy.
\end{align}
\end{lemma}
\begin{proof}
For $n= 0,1$,
we apply Lemma \ref{lem:YoWa:en} as $f=f$, $g=g$ and $h=\frac{e^s}{a}\partial_y N(e^{-s}v_y) +h$ with
$l = \frac{1}{2}$,
$m = \frac{3}{2}$,
$c_1(s) = r^2e^{-s}/a$,
$c_2(s) = r'/a$,
and
$c_4(s) = e^{-s}/a$.
Using \eqref{eq:c1prime} and \eqref{eq:c4prime},
we have
\begin{align}
    \frac{d}{ds}E_{11}^{(n)}(s)
    &=
    - \int_{\mathbb{R}} y^{2n} g^2 \,dy
    + \frac{3-2n}{4} \int_{\mathbb{R}} y^{2n} f_y^2 \,dy
    + \frac{5-2n}{4} \frac{e^{-s}}{a} \int_{\mathbb{R}} y^{2n} f_{yy}^2 \,dy \\
    &\quad 
    + \frac{5-2n}{4} \frac{r^2e^{-s}}{a}
    \int_{\mathbb{R}} y^{2n} g^2 \,dy
    - \frac{r'}{a} \int_{\mathbb{R}} y^{2n} g^2 \,dy 
    -2n \int_{\mathbb{R}} y^{2n-1} f_y g \,dy \\
    &\quad 
    -2n(2n-1) \frac{e^{-s}}{a} \int_{\mathbb{R}} y^{2n-2} f_{yy}g \,dy
    -4n \frac{e^{-s}}{a} \int_{\mathbb{R}} y^{2n-1} f_{yy} g_{y} \,dy \\
    &\quad
    -\frac{1}{2} \left( \frac{e^{-s}}{a} + \frac{a'}{ra^2} \right) \int_{\mathbb{R}} y^{2n} f_{yy}^2 \,dy
    + \frac{1}{2} \left( \frac{2r'}{a} - \frac{r^2e^{-s}}{a} - \frac{ra'}{a^2} \right) \int_{\mathbb{R}} y^{2n} g^2 \,dy \\
    &\quad
    + \int_{\mathbb{R}} y^{2n} g
    \left( \frac{e^{s}}{a} \partial_y (N(e^{-s} v_y)) + h \right) \,dy \\
    &=
    - \int_{\mathbb{R}} y^{2n} g^2 \,dy
    + \frac{3-2n}{2} E_{11}^{(n)}(s) \\
    &\quad
    -2n \int_{\mathbb{R}} y^{2n-1} f_y g \,dy
    -2n(2n-1) \frac{e^{-s}}{a} \int_{\mathbb{R}} y^{2n-2} f_{yy}g \,dy
    -4n \frac{e^{-s}}{a} \int_{\mathbb{R}} y^{2n-1} f_{yy} g_{y} \,dy \\
    &\quad
    -\frac{a'}{2ra^2} \int_{\mathbb{R}} y^{2n} f_{yy}^2 \,dy
    - \frac{ra'}{2a^2} \int_{\mathbb{R}} y^{2n} g^2 \,dy \\
    &\quad
    + \int_{\mathbb{R}} y^{2n} g
    \left( \frac{e^{s}}{a} \partial_y (N(e^{-s} v_y)) + h \right) \,dy.
\end{align}
Similarly, we calculate
\begin{align}
    \frac{d}{ds} E_{12}^{(n)}(s)
    &=
    - \int_{\mathbb{R}} y^{2n} f_y^2 \,dy
    - \frac{e^{-s}}{a} \int_{\mathbb{R}} y^{2n} f_{yy}^2 \,dy
    + \frac{1-2n}{4} \int_{\mathbb{R}} y^{2n} f^2 \,dy \\
    &\quad
    + \frac{r^2e^{-s}}{a} \int_{\mathbb{R}} y^{2n} g^2 \,dy
    + \frac{3-2n}{2} \frac{r^2e^{-s}}{a} \int_{\mathbb{R}} y^{2n} fg \,dy
    - \frac{r'}{a} \int_{\mathbb{R}} y^{2n} fg \,dy \\
    &\quad
    -2n \int_{\mathbb{R}} y^{2n-1} f f_y \,dy
    -4n \frac{e^{-s}}{a} \int_{\mathbb{R}} y^{2n-1} f_y f_{yy} \,dy
    -2n (2n-1) \frac{e^{-s}}{a} \int_{\mathbb{R}} y^{2n-2} f f_{yy} \,dy \\
    &\quad
    + \left( \frac{2r'}{a} - \frac{r^2e^{-s}}{a} - \frac{ra'}{a^2} \right) \int_{\mathbb{R}} y^{2n} fg \,dy
    + \int_{\mathbb{R}} y^{2n} f
    \left( \frac{e^{s}}{a} \partial_y (N(e^{-s}v_y)) + h \right) \,dy \\
    &=
    \frac{1-2n}{2} E_{12}^{(n)}(s)
    -2 E_{11}^{(n)}(s)
    + \frac{2r^2e^{-s}}{a} \int_{\mathbb{R}} y^{2n} g^2 \,dy\\
    &\quad
    -2n \int_{\mathbb{R}} y^{2n-1} f f_y \,dy
    -4n \frac{e^{-s}}{a} \int_{\mathbb{R}} y^{2n-1} f_y f_{yy} \,dy
    -2n (2n-1) \frac{e^{-s}}{a} \int_{\mathbb{R}} y^{2n-2} f f_{yy} \,dy \\
    &\quad
    + \left( \frac{r'}{a} - \frac{ra'}{a^2} \right) \int_{\mathbb{R}} y^{2n} fg \,dy
    + \int_{\mathbb{R}} y^{2n} f
    \left( \frac{e^{s}}{a} \partial_y (N(e^{-s}v_y)) + h \right) \,dy.
\end{align}
This completes the proof.
\end{proof}

Finally, to control the bad term
$\displaystyle -4 \frac{e^{-s}}{a}\int_{\mathbb{R}} y f_{yy} g_y \,dy$
in
$\dfrac{d}{ds}E_{11}^{(1)}(s)$,
we consider the energies
\begin{align}
    E_{21}(s)
    &=
    \frac{1}{2} \int_{\mathbb{R}}
    \left( f_{yy}^2 + \frac{e^{-s}}{a}f_{yyy}^2 + \frac{r^2e^{-s}}{a} g_y^2 \right)\,dy,\\
    E_{22}(s)
    &=
    \int_{\mathbb{R}}
    \left( \frac{1}{2} f_y^2 + \frac{r^2e^{-s}}{a} f_y g_y \right) \,dy.
\end{align}
Since
$(f_y, g_y)$
satisfies the equations
\begin{align}
    \left\{\begin{alignedat}{3}
    &(f_y)_s - \frac{y}{2} (f_y)_y - f_y = g_y,\\
    &\frac{r^2e^{-s}}{a}
    \left( (g_y)_s - \frac{y}{2}(g_y)_y -2 g_y
    \right) + g_y + \frac{r'}{a} g_y 
    =
    (f_y)_{yy} - \frac{e^{-s}}{a} (f_y)_{yyyy}
    + \frac{e^s}{a} \partial_y^2 N(e^{-s} v_y)) + h_y,
    \end{alignedat} \right.
\end{align}
we have the following lemma.
\begin{lemma}\label{lem:E2}
We have
\begin{align}
    \frac{d}{ds}E_{21}(s)
    + \int_{\mathbb{R}} g_y^2 \,dy
    &=
    \frac{5}{2} E_{21}(s)
    - \frac{a'}{2ra^2} \int_{\mathbb{R}} f_{yyy}^2 \,dy
    - \frac{ra'}{2a^2} \int_{\mathbb{R}} g_y^2 \,dy \\
    &\quad
    + \int_{\mathbb{R}} g_y
    \left( \frac{e^s}{a}\partial_y^2 (N(e^{-s}v_y)) + h_y \right) \,dy
\end{align}
and
\begin{align}
    \frac{d}{ds} E_{22}(s) + 2 E_{21}(s)
    &=
    \frac{3}{2} E_{22}(s) 
    + 2 \frac{r^2e^{-s}}{a} \int_{\mathbb{R}} g_y^2 \,dy
    + \left( \frac{r'}{a} - \frac{ra'}{a^2} \right)
    \int_{\mathbb{R}} f_y g_y \,dy \\
    &\quad
    + \int_{\mathbb{R}} f_y
    \left( \frac{e^{s}}{a} \partial_y^2 (N(e^{-s}v_y)) + h_y \right)\,dy.
\end{align}
\end{lemma}
\begin{proof}
Applying Lemma \ref{lem:YoWa:en} as $f=f_y$, $g=g_y$ and $h=\frac{e^s}{a}\partial_{yy} N(e^{-s}v_y) +h_y$ with
$l =1$,
$m =2$,
$n = 0$,
$c_1(s) = \frac{r^2e^{-s}}{a}$,
$c_2(s) = \frac{r'}{a}$,
and
$c_4(s) = \frac{e^{-s}}{a}$,
and also using
\eqref{eq:c1prime} and \eqref{eq:c4prime},
we have
\begin{align}
    \frac{d}{ds} E_{21}(s) 
    &=
    - \int_{\mathbb{R}} g_y^2 \,dy
    + \frac{5}{4} \int_{\mathbb{R}}  f_{yy}^2 \,dy
    + \frac{7}{4} \frac{e^{-s}}{a} \int_{\mathbb{R}} f_{yyy}^2 \,dy\\
    &\quad
    + \frac{7}{4} \frac{r^2 e^{-s}}{a} \int_{\mathbb{R}} g_y^2 \,dy
    - \frac{r'}{a} \int_{\mathbb{R}} g_y^2 \,dy \\
    &\quad
    - \frac{1}{2} \left( \frac{e^{-s}}{a} + \frac{a'}{ra^2} \right)
    \int_{\mathbb{R}} f_{yyy}^2 \,dy
    + \frac{1}{2} \left( \frac{2r'}{a}- \frac{r^2e^{-s}}{a} - \frac{ra'}{a^2} \right)
    \int_{\mathbb{R}} g_y^2 \,dy \\
    &\quad
    + \int_{\mathbb{R}} g_y
    \left( \frac{e^s}{a}\partial_y^2 (N(e^{-s}v_y)) + h_y \right) \,dy \\
    &=
    - \int_{\mathbb{R}} g_y^2 \,dy
    + \frac{5}{2} E_{21}(s) \\
    &\quad
    - \frac{a'}{2ra^2} \int_{\mathbb{R}} f_{yyy}^2 \,dy
    - \frac{ra'}{2a^2} \int_{\mathbb{R}} g_y^2 \,dy \\
    &\quad
    + \int_{\mathbb{R}} g_y
    \left( \frac{e^s}{a}\partial_y^2 (N(e^{-s}v_y)) + h_y \right) \,dy.
\end{align}
Similarly, we have
\begin{align}
    \frac{d}{ds} E_{22}(s)
    &=
    - \int_{\mathbb{R}} f_{yy}^2 \,dy
    - \frac{e^{-s}}{a} \int_{\mathbb{R}} f_{yyy}^2 \,dy
    + \frac{3}{4} \int_{\mathbb{R}} f_y^2 \,dy \\
    &\quad
    + \frac{r^2 e^{-s}}{a} \int_{\mathbb{R}} g_y^2 \,dy
    + \frac{5}{2} \frac{r^2 e^{-s}}{a}
    \int_{\mathbb{R}} f_y g_y \,dy
    - \frac{r'}{a} \int_{\mathbb{R}} f_y g_y \,dy \\
    &\quad
    + \left( \frac{2r'}{a} - \frac{r^2 e^{-s}}{a} - \frac{ra'}{a^2} \right)
    \int_{\mathbb{R}} f_y g_y \,dy
    + \int_{\mathbb{R}} f_y
    \left( \frac{e^{s}}{a} \partial_y^2 (N(e^{-s}v_y)) + h_y \right)\,dy \\
    &=
    \frac{3}{2} E_{22}(s) - 2 E_{21}(s) \\
    &\quad
    + 2 \frac{r^2e^{-s}}{a} \int_{\mathbb{R}} g_y^2 \,dy
    + \left( \frac{r'}{a} - \frac{ra'}{a^2} \right)
    \int_{\mathbb{R}} f_y g_y \,dy \\
    &\quad
    + \int_{\mathbb{R}} f_y
    \left( \frac{e^{s}}{a} \partial_y^2 (N(e^{-s}v_y)) + h_y \right)\,dy.
\end{align}
This completes the proof.
\end{proof}

Finally, we define
\begin{align}
    E_{m1}(s)
    &:= \frac{1}{2} \frac{r^2e^{-s}}{a} m_s(s)^2,\\
    E_{m2}(s)
    &:= \frac{1}{2} m(s)^2 + \frac{r^2e^{-s}}{a} m(s) m_s(s).
\end{align}
Then, we have the following energy identities.
\begin{lemma}\label{lem:Em}
We have
\begin{align}
    \frac{d}{ds} E_{m1}(s)
    + \frac{1}{2} E_{m1}(s) + m_s(s)^2
    &=
    \left(
    \frac{3r^2 e^{-s}}{4a} - \frac{r a'}{2a^2} 
    \right) m_s(s)^2
\end{align}
and
\begin{align}
    \frac{d}{ds} E_{m2}(s)
    &=
    2 E_{m1}(s) +
    \left( \frac{r'}{a} - \frac{ra'}{a^2}
    \right) m(s) m_s(s).
\end{align}
\end{lemma}
\begin{proof}
By \eqref{eq:c1prime} and Lemma \ref{lem:eq:m},
we have
\begin{align}
    \frac{d}{ds} E_{m1}(s)
    &=
    \frac{1}{2} \frac{d}{ds} \left( \frac{r^2e^{-s}}{a} \right) m_s^2
    + \frac{r^2 e^{-s}}{a} m_s m_{ss} \\
    &=
    \frac{1}{2} \left( \frac{2r'}{a} - \frac{r^2e^{-s}}{a} - \frac{ra'}{a^2} \right) m_s^2 \\
    &\quad
    + \frac{r^2e^{-s}}{a} m_s^2
    - \left( 1 + \frac{r'}{a} \right) m_s^2 \\
    &=
    \frac{1}{2} \frac{r^2e^{-s}}{a} m_s^2
    - \frac{1}{2} \frac{ra'}{a^2} m_s^2
    - m_s^2 \\
    &=
    - \frac{1}{2} E_{m1}(s) - m_s^2 
    + \left( \frac{3r^2e^{-s}}{4a} - \frac{ra'}{2a^2} \right) m_s^2.
\end{align}
Similarly, we have
\begin{align}
    \frac{d}{ds} E_{m2}(s)
    &=
    m m_s
    + \frac{d}{ds} \left( \frac{r^2e^{-s}}{a} \right) m m_s
    + \frac{r^2e^{-s}}{a} m_s^2
    + \frac{r^2e^{-s}}{a} m m_{ss} \\
    &=
    m m_s
    + \left( \frac{2r'}{a} - \frac{r^2e^{-s}}{a} - \frac{ra'}{a^2} \right) m m_s \\
    &\quad
    + \frac{r^2e^{-s}}{a} m_s^2
    + \frac{r^2e^{-s}}{a} m m_{s}
    - \left( 1 + \frac{r'}{a} \right) m m_s \\
    &=
    \frac{r^2e^{-s}}{a} m_s^2
    + \left( \frac{r'}{a} - \frac{ra'}{a^2} \right) m m_s \\
    &=
    2 E_{m1}(s)
    + \left( \frac{r'}{a} - \frac{ra'}{a^2} \right) m m_s.
\end{align}
\end{proof}

\subsection{Energy estimates}
Now, we combine the energy identities
in the previous subsection to obtain
the energy estimates.
First, we prepare the following estimates
for remainders.
\begin{lemma}\label{lem:remainder}
Set
\begin{align}
    \delta
    &:=
    \min \left\{ \frac{\beta + 1}{\alpha-\beta + 1},
    \frac{2\alpha-\beta + 1}{\alpha-\beta+1} \right\},
\end{align}
which is positive if
$(\alpha,\beta) \in \Omega_1$.
Then, we have
\begin{align}
    \frac{r^2 e^{-s}}{a}
    \sim 
        e^{-\frac{\beta+1}{\alpha-\beta+1}s}
    \le
    C e^{-\delta s},
    \quad
    \frac{e^{-s}}{a}
    \sim
        e^{-\frac{2\alpha-\beta+1}{\alpha-\beta+1}s}
    \le
    C e^{-\delta s},
\end{align}
and
\begin{align} 
    \left| \frac{a'}{ra^2} \right|
    \le 
    C \frac{e^{-s}}{a},
    \quad
    \left| \frac{ra'}{a^2} \right|
    \le
    C \frac{r^2e^{-s}}{a},
    \quad
    \left| \frac{r'}{a} \right|
    \le C \frac{r^2e^{-s}}{a}.
\end{align}
\end{lemma}

Define
\begin{align}
    \mathbb{E}_0(s)
    &:=
    E_{01}(s) + c_0 E_{02}(s),
\end{align}
where
$c_0 > 0$
is a sufficiently large constant determined later.

\begin{lemma}\label{lem:EE0}
There exists a constant
$c_0 > 0$
satisfying the following:
For any
$\eta > 0$,
there exists $s_0 > 0$ such that
for any $s \ge s_0$, we have
\begin{align}
	\mathbb{E}_0(s)
	\ge
	C \left( \int_{\mathbb{R}} F_y^2 \,dy 
	+ \frac{e^{-s}}{2a}
    \int_{\mathbb{R}} F_{yy}^2 \,dy 
    + \frac{r^2e^{-s}}{2a}
    \int_{\mathbb{R}} G^2 \,dy
    + \int_{\mathbb{R}} F^2 \,dy \right)
\end{align}
and
\begin{align}
    &\frac{d}{ds} \mathbb{E}_0(s)
    +\frac{1}{2} \mathbb{E}_0(s)
    + \frac{1}{4} \int_{\mathbb{R}} G^2 \,dy \\
    &\le
    \eta \mathbb{E}_0(s)
    + C(\eta) \left(
    \| H(s) \|_{L^2}^2
    + \frac{e^{2s}}{a^2}
    \| N(e^{-s} v_y) \|_{L^2}^2
    \right).
\end{align}
\end{lemma}
\begin{proof}
Let
$\eta > 0$ be arbitrary.
Lemmas \ref{lem:E0} and \ref{lem:remainder},
and the Schwarz inequality imply
\begin{align}
    \frac{d}{ds}E_{01}(s) + \int_{\mathbb{R}} G^2 \,dy
    &\le
    \frac{1}{2} E_{01}(s)
    + C_1 \frac{e^{-s}}{a} \int_{\mathbb{R}} F_{yy}^2 \,dy
    + C_1 \frac{r^2e^{-s}}{a} \int_{\mathbb{R}} G^2 \,dy \\
    &\quad
    + \frac{1}{2} \int_{\mathbb{R}} G^2 \,dy
    + C \left(
    \| H(s) \|_{L^2}^2
    + \frac{e^{2s}}{a^2}
    \| N(e^{-s} v_y) \|_{L^2}^2
    \right) \\
    &\le
    \left( \frac{1}{2} + 2C_1 \right) E_{01}(s)
    + \frac{1}{2} \int_{\mathbb{R}} G^2 \,dy \\
    &\quad
    + C \left(
    \| H(s) \|_{L^2}^2
    + \frac{e^{2s}}{a^2}
    \| N(e^{-s} v_y) \|_{L^2}^2
    \right)
\end{align}
with some $C_1 > 0$ and
\begin{align}
    \frac{d}{ds} E_{02}(s) + \frac{1}{2} E_{02}(s)
    + 2E_{01}(s)
    &\le
    C_2(\eta_1) e^{-\delta s} \int_{\mathbb{R}} G^2 \,dy
    + \eta_1 \int_{\mathbb{R}} F^2 \,dy \\
    &\quad
    + C(\eta_1)  \left(
    \| H(s) \|_{L^2}^2
    + \frac{e^{2s}}{a^2}
    \| N(e^{-s} v_y) \|_{L^2}^2
    \right),
\end{align}
with some $C_2(\eta_1) > 0$,
where
$\eta_1$
is an arbitrary small positive number determined later.
We take
$c_0$
sufficiently large so that
$2c_0 - \frac{1}{2} -2C_1 \ge \frac{1}{2}$.
Then, letting
$s_0$ sufficiently large so that
$c_0 C_2(\eta_1) e^{-\delta s} \le \frac{1}{4}$
holds for any $s \ge s_0$,
we conclude
\begin{align}
    &\frac{d}{ds} \mathbb{E}_0(s)
    +\frac{1}{2} \mathbb{E}_0(s)
    + \frac{1}{4} \int_{\mathbb{R}} G^2 \,dy \\
\label{eq:lem:E0}
    &\le
    2c_0 \eta_1 \int_{\mathbb{R}} F^2 \,dy
    + C \left(
    \| H(s) \|_{L^2}^2
    + \frac{e^{2s}}{a^2}
    \| N(e^{-s} v_y) \|_{L^2}^2
    \right).
\end{align}
On the other hand, we remark that
\begin{align}
	 \frac{r^2e^{-s}}{a} \left| \int_{\mathbb{R}} F(s,y) G(s,y) \,dy \right|
	 &\le
	 \frac{1}{4} \int_{\mathbb{R}} F(s,y)^2 \,dy
	 + C \left( \frac{r^2e^{-s}}{a} \right)^2  \int_{\mathbb{R}} G(s,y)^2 \,dy \\
	 &\le 
	 \frac{1}{4} \int_{\mathbb{R}} F(s,y)^2 \,dy
	 + C e^{-\delta s} \frac{r^2e^{-s}}{2a}
    \int_{\mathbb{R}} G(s,y)^2 \,dy.
\end{align}
From this, retaking $s_0$ larger if needed, we have for $s \ge s_0$,
\begin{align}
	\mathbb{E}_0(s)
	\ge
	C \left( \int_{\mathbb{R}} F_y^2 \,dy 
	+ \frac{e^{-s}}{2a}
    \int_{\mathbb{R}} F_{yy}^2 \,dy 
    + \frac{r^2e^{-s}}{2a}
    \int_{\mathbb{R}} G^2 \,dy
    + \int_{\mathbb{R}} F^2 \,dy \right),
\end{align}
which shows the first assertion.
In particular, it gives
$\int_{\mathbb{R}} F_y^2 \,dy \le C \mathbb{E}_0(s)$
Applying this to the right-hand side of \eqref{eq:lem:E0}
and
taking $\eta_1$ so that $\eta  = 2 c_0 C \eta_1 $, 
we have the desired estimate.
\end{proof}

Next, for $n=0, 1$,
we define
\begin{align}
    \mathbb{E}^{(n)}_{1}(s)
    :=
    E^{(n)}_{11}(s) + c_1^{(n)} E^{(n)}_{12}(s),
\end{align}
where
$c_1^{(0)}$ and $c_1^{(1)}$
are sufficiently large constants
determined later.
The following two lemmas are
the estimates for
$\mathbb{E}_{1}^{(0)}(s)$
and
$\mathbb{E}_{1}^{(1)}(s)$,
respectively.

\begin{lemma}\label{lem:EE1:0}
There exist positive constants
$c_1^{(0)}$ 
and $s_1^{(0)}$
such that for any
$s \ge s_1^{(0)}$,
we have
\begin{align}
	\mathbb{E}^{(0)}_1(s)
	\ge
	C \left( \int_{\mathbb{R}} f_y^2 \,dy 
	+ \frac{e^{-s}}{2a}
    \int_{\mathbb{R}} f_{yy}^2 \,dy 
    + \frac{r^2e^{-s}}{2a}
    \int_{\mathbb{R}} g^2 \,dy
    + \int_{\mathbb{R}} f^2 \,dy \right)
\end{align}
and
\begin{align}
    &\frac{d}{ds} \mathbb{E}^{(0)}_1 (s)
    + \frac{1}{2} \mathbb{E}^{(0)}_1(s)
    + \frac{1}{4} \int_{\mathbb{R}} g^2 \,dy \\
    &\le
    C \mathbb{E}_0(s)
    + C \left(
    \| h(s) \|_{L^2}^2
    + \frac{e^{2s}}{a^2}
    \left\| \partial_y N \left(
        e^{-s} v_y
            \right) 
    \right\|_{L^2}^2
    \right).
\end{align}
\end{lemma}
\begin{proof}
By Lemmas \ref{lem:E1} and \ref{lem:remainder},
and the Schwarz inequality,
we have
\begin{align}
    &\frac{d}{ds} E_{11}^{(0)}(s)
    + \frac{1}{2} E_{11}^{(0)}(s)
    + \int_{\mathbb{R}} g^2\,dy \\
    &\le
    2E_{11}^{(0)}(s)
    + C \frac{e^{-s}}{a} \int_{\mathbb{R}} f_{yy}^2 \,dy
    + C \frac{r^2e^{-s}}{a} \int_{\mathbb{R}} g^2 \,dy \\
    &\quad
    + \frac{1}{2} \int_{\mathbb{R}} g^2 \,dy
    + C \left(
        \| h(s) \|_{L^2}^2
        + \frac{e^{2s}}{a^2}
        \| \partial_y \left( N \left( e^{-s} v_y \right) \right) \|_{L^2}^2
    \right),
\end{align}
which implies
\begin{align}
    &\frac{d}{ds} E_{11}^{(0)}(s)
    + \frac{1}{2} E_{11}^{(0)}(s)
    + \frac{1}{2}\int_{\mathbb{R}} g^2\,dy \\
    &\le
    \left( 2 + C_1 \right) E_{11}^{(0)}(s)
    + C \left(
        \| h(s) \|_{L^2}^2
        + \frac{e^{2s}}{a^2}
        \| \partial_y \left( N \left( e^{-s} v_y \right) \right) \|_{L^2}^2
    \right) 
\end{align}
with some constant
$C_1>0$.
In a similar way, we also obtain
\begin{align}
    &\frac{d}{ds} E_{12}^{(0)} (s)
    + \frac{1}{2} E_{12}^{(0)} (s)
    + 2 E_{11}^{(0)} (s) \\
    &\le
    C \int_{\mathbb{R}} f^2 \,dy
    + C_2 e^{-\delta s} \int_{\mathbb{R}} g^2 \,dy \\
    &\quad
    + C \left(
        \| h(s) \|_{L^2}^2
        + \frac{e^{2s}}{a^2}
        \| \partial_y \left( N \left( e^{-s} v_y \right)  \right) \|_{L^2}^2
    \right) 
\end{align}
with some constant
$C_2 > 0$.
Therefore, taking
$c_1^{(0)}$ and $s_1^{(0)}$
sufficiently large so that
$2 c_1^{(0)} - \left( 2 + C_1 \right) \ge \frac{1}{2}$
and
$c_1^{(0)} C_2 e^{-\delta s} \le \frac{1}{4}$
holds for any
$s \ge s_1^{(0)}$,
we conclude
\begin{align}
    &\frac{d}{ds} \mathbb{E}^{(0)}_1 (s)
    + \frac{1}{2} \mathbb{E}^{(0)}_1(s)
    + \frac{1}{4} \int_{\mathbb{R}} g^2 \,dy \\
    &\le
    C \int_{\mathbb{R}} f^2 \,dy
    + C \left(
    \| h(s) \|_{L^2}^2
    + \frac{e^{2s}}{a^2}
    \left\| \partial_y N \left(
        e^{-s} v_y
            \right) 
    \right\|_{L^2}^2
    \right).
\end{align}
Finally, by $\int_{\mathbb{R}} f^2 \,dy = \int_{\mathbb{R}} F_y^2 \,dy \le C \mathbb{E}_0$,
the proof of the second assertion is complete.
The first assertion is proved in the same way as the previous lemma and we omit the detail.
\end{proof}
\begin{lemma}\label{lem:EE1:1}
There exists a constant $c_1^{(1)} > 0$ satisfying the following:
for any $\eta' > 0$, there exists a constant $s_1^{(1)} > 0$ such that
for any $s \ge s_1^{(1)}$, we have
\begin{align}
	\mathbb{E}^{(1)}_1(s)
	\ge
	C \left( \int_{\mathbb{R}} y^2f_y^2 \,dy 
	+ \frac{e^{-s}}{2a}
    \int_{\mathbb{R}} y^2 f_{yy}^2 \,dy 
    + \frac{r^2e^{-s}}{2a}
    \int_{\mathbb{R}} y^2 g^2 \,dy
    + \int_{\mathbb{R}} y^2 f^2 \,dy \right)
\end{align}
and
\begin{align}
    &\frac{d}{ds} \mathbb{E}_{1}^{(1)} (s) + \frac{1}{2} \mathbb{E}_1^{(1)} (s)
        + \frac{1}{4} \int_{\mathbb{R}} y^2 g^2 \,dy \\
    &\le
    \eta' \mathbb{E}_1^{(1)}(s)
    + C \mathbb{E}_1^{(0)} (s) 
    - 4 \frac{e^{-s}}{a} \int_{\mathbb{R}} y f_{yy} g_y \,dy
   + C e^{-\delta s} \int_{\mathbb{R}} g^2 dy 
    \\
    &\quad
    + C(\eta')
    \left(
    \| y h(s) \|_{L^2}^2 + \frac{e^{2s}}{a^2} \| y \partial_y (N(e^{-s}v_y)) \|_{L^2}^2 
    \right).
\end{align}
\end{lemma}
\begin{proof}
Let $\eta' > 0$ be arbitrary.
By Lemmas \ref{lem:E1} and \ref{lem:remainder},
and the Schwarz inequality, we have
\begin{align}
    &\frac{d}{ds} E_{11}^{(1)}(s)
    + \frac{1}{2} E_{11}^{(1)}(s) 
    +  \int_{\mathbb{R}} y^2 g^2 \,dy \\
    &\le
    E_{11}^{(1)}(s)
    + C \frac{e^{-s}}{a} \int_{\mathbb{R}} y^2 f_{yy}^2 \,dy
    + C \frac{e^{-s}}{a} \int_{\mathbb{R}} f_{yy}^2 \,dy
    + C \frac{r^2 e^{-s}}{a} \int_{\mathbb{R}} y^2 g^2 \,dy \\
    &\quad
    + \frac{1}{2} \int_{\mathbb{R}} y^2 g^2 \,dy
    + C \int_{\mathbb{R}} f_y^2 \,dy
    + C e^{-\delta s} \int_{\mathbb{R}} g^2 \,dy \\
    &\quad
    - 4 \frac{e^{-s}}{a} \int_{\mathbb{R}} y f_{yy} g_y \,dy
    + C
    \left(
    \| y h(s) \|_{L^2}^2 + \frac{e^{2s}}{a^2} \| y \partial_y (N(e^{-s}v_y)) \|_{L^2}^2 
    \right),
\end{align}
which implies
\begin{align}
     &\frac{d}{ds} E_{11}^{(1)}(s)
    + \frac{1}{2} E_{11}^{(1)}(s) 
    +  \frac{1}{2} \int_{\mathbb{R}} y^2 g^2 \,dy \\
    &\le
    (1+C_1') E_{11}^{(1)}(s)
    + C \mathbb{E}_1^{(0)}(s)
    + C e^{-\delta s} \int_{\mathbb{R}} g^2 \,dy \\
    &\quad
    - 4 \frac{e^{-s}}{a} \int_{\mathbb{R}} y f_{yy} g_y \,dy
    + C
    \left(
    \| y h(s) \|_{L^2}^2 + \frac{e^{2s}}{a^2} \| y \partial_y (N(e^{-s}v_y)) \|_{L^2}^2 
    \right)
\end{align}
with some constant
$C_1'>0$.
Next, for $E_{12}^{(1)}(s)$,
Lemmas \ref{lem:E1} and \ref{lem:remainder} and the Schwarz inequality imply
\begin{align}
    &\frac{d}{ds} E_{12}^{(1)} (s)
    + \frac{1}{2} E_{12}^{(1)}(s)
    + 2 E_{11}^{(1)}(s) \\
    &=
    \frac{2r^2e^{-s}}{a} \int_{\mathbb{R}} y^{2} g^2 \,dy
    -2 \int_{\mathbb{R}} y f f_y \,dy
    -4 \frac{e^{-s}}{a} \int_{\mathbb{R}} y f_y f_{yy} \,dy
    -2 \frac{e^{-s}}{a} \int_{\mathbb{R}} f f_{yy} \,dy \\
    &\quad
    + \left( \frac{r'}{a} - \frac{ra'}{a^2} \right) \int_{\mathbb{R}} y^{2} fg \,dy
    + \int_{\mathbb{R}} y^{2} f
    \left( \frac{e^{s}}{a} \partial_y (N(e^{-s}v_y)) + h \right) \,dy \\
    &\le
    \eta_1' \int_{\mathbb{R}} y^2 f^2 \,dy
    + E_{11}^{(1)}(s)
    + C_2'(\eta_1') e^{-\delta s} \int_{\mathbb{R}} y^2 g^2 \,dy \\
    &\quad
    + C \int_{\mathbb{R}} f_y^2 \,dy
    + C \left( \frac{e^{-s}}{a} \right)^2 \int_{\mathbb{R}} f_{yy}^2 \,dy
    + C \int_{\mathbb{R}} f^2 \,dy \\
    &\quad
    + C
    \left(
    \| y h(s) \|_{L^2}^2 + \frac{e^{2s}}{a^2} \| y \partial_y (N(e^{-s}v_y)) \|_{L^2}^2 
    \right) \\
    &\le
    \eta_1' \int_{\mathbb{R}} y^2 f^2 \,dy
    + E_{11}^{(1)}(s)
    + C_2'(\eta_1') e^{-\delta s} \int_{\mathbb{R}} y^2 g^2 \,dy
    + C \mathbb{E}_1^{(0)}(s) \\
    &\quad
    + C
    \left(
    \| y h(s) \|_{L^2}^2 + \frac{e^{2s}}{a^2} \| y \partial_y (N(e^{-s}v_y)) \|_{L^2}^2 
    \right)
\end{align}
for arbitrary small $\eta_1' > 0$ determined later and some constant $C_2'(\eta_1')>0$.
Therefore, taking $c_1^{(1)}$ and $s_1^{(1)}$ so that
$c_1^{(1)} - (1+C_1') \ge \frac{1}{2}$
and
$c_1^{(1)} C_2'(\eta_1') e^{-\delta s} \le \frac{1}{4}$
holds for any $s \ge s_1^{(1)}$,
we conclude
\begin{align}
    &\frac{d}{ds} \mathbb{E}_{1}^{(1)} + \frac{1}{2} \mathbb{E}_1^{(1)} + \frac{1}{4} \int_{\mathbb{R}} y^2 g^2 \,dy \\
    &\le
    \eta_1' c_1^{(1)} \int_{\mathbb{R}} y^2 f^2 \,dy
    - 4 \frac{e^{-s}}{a} \int_{\mathbb{R}} y f_{yy} g_y \,dy
    + C \int_{\mathbb{R}} f_y^2 \,dy
    + C e^{-\delta s} \int_{\mathbb{R}} g^2 \,dy \\
    &\quad
    + C \mathbb{E}_1^{(0)}(s)
    + C
    \left(
    \| y h(s) \|_{L^2}^2 + \frac{e^{2s}}{a^2} \| y \partial_y (N(e^{-s}v_y)) \|_{L^2}^2 
    \right).
\end{align}
Taking $\eta_1'$ so that the first term of the right-hand side is bounded by
$\eta' \mathbb{E}_1^{(1)}(s)$
and using
$\int_{\mathbb{R}} f_y^2 \,dy \le C \mathbb{E}_1^{(0)}(s)$,
we complete the proof
of the second assertion.
The first assertion is proved in the same way as before and we omit the detail.
\end{proof}
Next, we define
\begin{align}
     \mathbb{E}_2(s)
    := E_{21}(s) + c_2 E_{22}(s),
\end{align}
where
$c_2$
is a sufficiently large constant determined later.
\begin{lemma}\label{lem:EE2}
There exist positive constants $c_2$ and $s_2$
such that for any $s \ge s_2$, we have
\begin{align}
	\mathbb{E}_2(s)
	\ge
	C \left( \int_{\mathbb{R}} f_{yy}^2 \,dy 
	+ \frac{e^{-s}}{2a}
    \int_{\mathbb{R}} f_{yyy}^2 \,dy 
    + \frac{r^2e^{-s}}{2a}
    \int_{\mathbb{R}} g_y^2 \,dy
    + \int_{\mathbb{R}} f_y^2 \,dy \right)
\end{align}
and
\begin{align}
    &\frac{d}{ds} \mathbb{E}_2(s)
    + \frac{1}{2} \mathbb{E}_2(s)
    + \frac{1}{4} \int_{\mathbb{R}} g_y^2 \,dy \\
    &\le
    C \mathbb{E}_{1}^{(0)} (s)
    + C
    \left(
    \| \partial_y h(s) \|_{L^2}^2 + \frac{e^{2s}}{a^2} \| \partial_y^2 (N(e^{-s}v_y)) \|_{L^2}^2 
    \right).
\end{align}
\end{lemma}
\begin{proof}
By Lemmas \ref{lem:E2} and \ref{lem:remainder} and the Schwarz inequality, we have
\begin{align}
    \frac{d}{ds} E_{21}(s) + \int_{\mathbb{R}} g_y^2 \,dy
    &\le
    \frac{5}{2} E_{21}(s) + C_1 \frac{e^{-s}}{a} \int_{\mathbb{R}} f_{yyy}^2 \,dy
    + C_1 \frac{r^2 e^{-s}}{a} \int_{\mathbb{R}} g_y^2 \,dy \\
    &\quad
    + \frac{1}{2} \int_{\mathbb{R}} g_y^2 \,dy
    + C \left( \| h_y(s) \|_{L^2}^2 + \frac{e^{2s}}{a^2} \| \partial_y^2 N (e^{-s}v_y) \|_{L^2}^2 \right) \\
    &\le
    \left( \frac{5}{2} + 2C_1 \right) E_{21}(s)
    + \frac{1}{2} \int_{\mathbb{R}} g_y^2 \,dy \\
    &\quad
    + C \left( \| h_y(s) \|_{L^2}^2 + \frac{e^{2s}}{a^2} \| \partial_y^2 N (e^{-s}v_y) \|_{L^2}^2 \right).
\end{align}
with some $C_1 > 0$ and 
\begin{align}
    &\frac{d}{ds} E_{22}(s) + \frac{1}{2} E_{22}(s) + 2 E_{21}(s) \\
    &\le 
    2E_{22}(s) + C \frac{r^2 e^{-s}}{a} \int_{\mathbb{R}} g_y^2 \,dy
    + C \frac{r^2 e^{-s}}{a} \int_{\mathbb{R}} f_y^2 \,dy \\
    &\quad
    + C \int_{\mathbb{R}} f_y^2 \,dy
    + C \left( \| h_y(s) \|_{L^2}^2 + \frac{e^{2s}}{a^2} \| \partial_y^2 N (e^{-s}v_y) \|_{L^2}^2 \right) \\
    &\le
    C_2 e^{-\delta s} \int_{\mathbb{R}} g_y^2 \,dy
    + C \int_{\mathbb{R}} f_y^2 \,dy \\
    &\quad
    + C \left( \| h_y(s) \|_{L^2}^2 + \frac{e^{2s}}{a^2} \| \partial_y^2 N (e^{-s}v_y) \|_{L^2}^2 \right)
\end{align}
with some $C_2>0$.
We take
$c_2$
sufficiently large so that
$2 c_2 - \frac{5}{2} - 2C_1 \ge \frac{1}{2}$.
Then, letting
$s_2$
sufficiently large so that 
$c_2 C_2 e^{-\delta s} \le \frac{1}{4}$
holds for any $s \ge s_2$,
we conclude
\begin{align}
    &\frac{d}{ds} \mathbb{E}_{2}(s) + \frac{1}{2} \mathbb{E}_2(s) + \frac{1}{4}\int_{\mathbb{R}} g_y^2 \,dy \\
    &\le 
    C \int_{\mathbb{R}} f_y^2 \,dy
    + C \left( \| h_y(s) \|_{L^2}^2 + \frac{e^{2s}}{a^2} \| \partial_y^2 N (e^{-s}v_y) \|_{L^2}^2 \right)
\end{align}
for any $s \ge s_2$.
This and
$\int_{\mathbb{R}} f_y^2 \,dy \le C \mathbb{E}_{1}^{(0)}(s)$
complete the proof
of the second assertion.
The first assertion is proved in the same way as before and we omit the detail.
\end{proof}

Finally, let us combine the estimates in Lemmas \ref{lem:EE0}--\ref{lem:EE2}.
Fix 
\begin{align}
    \lambda \in \left( 0, \min \left\{ \frac{1}{2}, \frac{2(\beta+1)}{\alpha-\beta+1}, \frac{2\alpha -\beta +1}{\alpha - \beta +1} \right\}  \right)
\end{align}
and let
$s_{\ast}' = \max\{ s_0, s_1^{(0)}, s_1^{(1)}, s_2\}$.
We first note that
the Schwarz inequality and Lemma \ref{lem:remainder} imply
\begin{align}
    -4 \frac{e^{-s}}{a} \int_{\mathbb{R}} y f_{yy} g_y \,dy
    &\le
    \eta''  \mathbb{E}_1^{(1)}(s)
    + C(\eta'') e^{-  \delta s} \int_{\mathbb{R}} g_y^2 \,dy
\end{align}
for any $\eta'' > 0$.
We take
$\eta, \eta_1'$ in Lemmas \ref{lem:EE0} and \ref{lem:EE1:1} and $\eta''$ above so that
$\frac{1}{2} - \eta \le \lambda$
and
$\eta'+\eta'' \le \frac{1}{2} - \lambda$.
Then, we take
$\tilde{c}_0 \gg \tilde{c}_1^{(0)} \gg \tilde{c}_1^{(1)} \gg 1$
and define
\begin{align}
    \mathcal{E}(s) 
    &=
    \tilde{c}_0 \mathbb{E}_0(s) + \tilde{c}_1^{(0)} \mathbb{E}_1^{(0)}(s)
    + \tilde{c}_1^{(1)} \mathbb{E}_1^{(1)}(s) + \mathbb{E}_2(s) + E_{m1}(s),\\
    \mathcal{G}(s)
    &=
    \int_{\mathbb{R}}
    \left(
    \tilde{c}_0 G^2 + \tilde{c}_1^{(0)} g^2 + \tilde{c}_1^{(1)} y^2 g^2 + g_y^2
    \right) \,dy,\\
    \widetilde{\mathcal{E}}(s)
    &=
    \mathcal{E}(s) + E_{m2}(s).
\end{align}
Then, adding the estimates in Lemmas \ref{lem:Em} and \ref{lem:EE0}--\ref{lem:EE2},
we conclude that
\begin{align}
    &\frac{d}{ds} \mathcal{E}(s) + \lambda \mathcal{E}(s)
    + \frac{1}{4} \mathcal{G}(s) + m_s(s)^2 \\
    &\le 
    C e^{-\delta s} \int_{\mathbb{R}} g^2 \, dy 
    + C e^{-\delta s} \int_{\mathbb{R}} g_y^2 \,dy
    + \left( \frac{3r^2e^{-s}}{4a} - \frac{ra'}{2a^2} \right) m_s(s)^2 \\
    &\quad
    + C \left( \| H(s) \|_{L^2}^2 + \| h(s) \|_{H^{0,1}}^2 + \| h_y (s) \|_{L^2}^2 \right) \\
    &\quad
    + C \left( \frac{e^s}{a} \right)^2
    \left(
    \| N(e^{-s} v_y) \|_{L^2}^2 + \left\| \partial_y N \left(e^{-s} v_y\right)  \right\|_{H^{0,1}}^2
    + \| \partial_y^2 (N(e^{-s}v_y)) \|_{L^2}^2 
    \right)
\end{align}
holds for $s \ge s_{\ast}'$.
Moreover, Lemma \ref{lem:remainder} leads to
\begin{align}
    \left( \frac{3r^2e^{-s}}{4a} - \frac{ra'}{2a^2} \right)
    \le 
    Ce^{-\delta s}.
\end{align}
Therefore, we finally reach the following energy estimate.
\begin{proposition}\label{prop:energyest}
There exist constants $s_{\ast} > 0$ and $C > 0$ such that for any $s \ge s_{\ast}$,
we have
\begin{align}
    &\frac{d}{ds} \mathcal{E}(s) + \lambda \mathcal{E}(s)
    + \frac{1}{8} \left( \mathcal{G}(s) + m_s^2 \right)\\
    &\le
    C \left( \| H(s) \|_{L^2}^2 + \| h(s) \|_{H^{0,1}}^2 + \| h_y (s) \|_{L^2}^2 \right) \\
    &\quad
    + C \left( \frac{e^s}{a} \right)^2
    \left(
    \| N(e^{-s} v_y) \|_{L^2}^2 + \left\| \partial_y N \left(e^{-s} v_y\right)  \right\|_{H^{0,1}}^2
    + \| \partial_y^2 (N(e^{-s}v_y)) \|_{L^2}^2 
    \right).
\end{align}
\end{proposition}

\section{Estimates of remainder terms and the proof of a priori estimate}
In this section, we give estimates of the right-hand side of Proposition \ref{prop:energyest},
and complete the a priori estimate, which ensures the existence of the global solution.

\subsection{Estimates of remainder terms}
First, by the Hardy-type inequality in Lemma \ref{lem:hardy}, we have
\begin{align}
    \| H(s) \|_{L^2}^2 \le 4 \| y h(s) \|_{L^2}^2,\quad
    \| N(e^{-s} v_y) \|_{L^2}^2 \le 4 \| y \partial_y N(e^{-s} v_y) \|_{L^2}^2.
\end{align}
Hence, it suffices to estimate
\begin{align}
    \| h(s) \|_{H^{0,1}}^2,\quad
    \| h_y (s) \|_{L^2}^2, \quad 
    \left( \frac{e^s}{a} \right)^2 \left\| \partial_y N \left(e^{-s} v_y\right)  \right\|_{H^{0,1}}^2,\quad
    \left( \frac{e^s}{a} \right)^2 \| \partial_y^2 (N(e^{-s}v_y)) \|_{L^2}^2.
\end{align}
First, from the definition of $h$ (see \eqref{eq:h}) and Lemma \ref{lem:remainder},
we easily obtain
\begin{align}\label{eq:4:est:h}
    \| h(s) \|_{H^{0,1}}^2 + \| h_y(s) \|_{L^2}^2
    &\le
    C e^{-2\delta s} \left( m(s)^2 + m_s(s)^2 \right) \le e^{-2\delta s}\widetilde{\mathcal{E}}(s).
\end{align}

Next, we estimate the nonlinear term.
By Assumption (N), we see that
\begin{align}
    \partial_y N (e^{-s} v_y) 
    &=
    2\mu e^{-2s} v_y v_{yy} + \tilde{N}'(e^{-s} v_y) e^{-s}v_{yy},\\
    \partial_y^2 N (e^{-s} v_y) 
    &=
    2\mu e^{-2s} (v_{yy}^2 + v_y v_{yyy} )
    + \tilde{N}''(e^{-s}v_y) e^{-2s} v_{yy}^2
    + \tilde{N}' (e^{-s} v_y) e^{-s} v_{yyy}.
\end{align}
Therefore, by
$|\tilde{N}'(z)|\le C|z|^{p-1}$,
the Sobolev embedding theorem,
and
\[
	\| f \|_{H^{2,1}}^2
	\le C \left( \frac{e^{-s}}{a} \right)^{-1} 
	 \big(\mathbb{E}_1^{(0)}(s)
    +  \mathbb{E}_1^{(1)}(s) \big)	
	\le C e^{\frac{2\alpha-\beta+1}{\alpha-\beta+1}s} \widetilde{\mathcal{E}}(s),
\]
we have
\begin{align} \label{eq:4:est:pyN}
    &\left( \frac{e^s}{a} \right)^2 \left\| \partial_y N \left(e^{-s} v_y\right)  \right\|_{H^{0,1}}^2 \\
    &\le
    C \left( \frac{e^s}{a} \right)^2
        e^{-4s} \| v_y v_{yy} \|_{H^{0,1}}^2
    + C \left( \frac{e^s}{a} \right)^2 e^{-2ps}
        \| |v_y|^{p-1} v_{yy} \|_{H^{0,1}}^2 \\
    &\le
    C e^{-2(1+\frac{\alpha}{\alpha-\beta+1})s} \| v_y \|_{L^{\infty}}^2 \| v_{yy} \|_{H^{0,1}}^2
    + C e^{-2(p-2)s} e^{-2(1+\frac{\alpha}{\alpha-\beta+1})s}
    		 \| v_y \|_{L^{\infty}}^{2(p-1)} \| v_{yy} \|_{H^{0,1}}^2 \\
    &\le
    C e^{-\frac{2(2\alpha -\beta +1)}{\alpha - \beta +1}s }
    \| v_y \|_{H^1}^2 \| v_{yy} \|_{H^{0,1}}^2 
    + C  e^{-2(p-2)s} e^{-\frac{2(2\alpha -\beta +1)}{\alpha - \beta +1}s } 
    		\| v_y \|_{H^1}^{2(p-1)} \| v_{yy} \|_{H^{0,1}}^2 \\
    &\le
    C \left(
    e^{-\frac{2(2\alpha -\beta +1)}{\alpha - \beta +1}s}  
    		( \| f \|_{H^{2,0}}^2 + m(s)^2 )
    +  e^{-2(p-2)s} e^{-\frac{2(2\alpha -\beta +1)}{\alpha - \beta +1}s } 
    		 ( \| f \|_{H^{2,0}}^2 + m(s)^2 )^{p-1}
    \right)\\
    &\quad
    \times ( \| f \|_{H^{2,1}}^2 + m(s)^2 ) \\
    & \le
    C e^{-\frac{2\alpha -\beta +1}{\alpha - \beta +1}s } \widetilde{\mathcal{E}}(s)^2
    + e^{-\left[ 2(p-2) + \frac{2\alpha -\beta +1}{\alpha - \beta +1} \right]s } \widetilde{\mathcal{E}}(s)^p.
\end{align}
Similarly, by $|\tilde{N}''(z)| \le C |z|^{p-2}$,
the Sobolev embedding theorem, and
\[
	\| f \|_{H^{3,0}}^2
	\le C \left( \frac{e^{-s}}{a} \right)^{-1}
	 \big(\mathbb{E}_1^{(0)}(s)
    +  {E}_{21}(s) \big)
	\le C e^{\frac{2\alpha-\beta+1}{\alpha-\beta+1}s} \widetilde{\mathcal{E}}(s),
\]
we obtain
\begin{align}\label{eq:4:est:pyyN}
    &\left( \frac{e^s}{a} \right)^2 \left\| \partial_y^2 N \left(e^{-s} v_y\right)  \right\|_{L^2}^2 \\
    &\le
    C \left( \frac{e^s}{a} \right)^2 e^{-4s}
        \left( \| v_{yy}^2 \|_{L^2}^2 + \| v_y v_{yyy} \|_{L^2}^2 \right) \\
    &\quad
    + C \left( \frac{e^s}{a} \right)^2 e^{-2ps}
    \left(
    \| |v_y|^{p-2} v_{yy}^2 \|_{L^2}^2
    + \| |v_y|^{p-1} v_{yyy} \|_{L^2}^2
    \right) \\
    &\le
    C e^{-2 (1+\frac{\alpha}{\alpha-\beta+1})s }
    ( \| v_{yy} \|_{L^{\infty}}^2 \| v_{yy} \|_{L^2}^2
    + \| v_y \|_{L^{\infty}}^2 \| v_{yyy} \|_{L^2}^2 ) \\
    &\quad
    + C  e^{-2(p-2)s} e^{-2 (1+\frac{\alpha}{\alpha-\beta+1})s } 
        ( \| v_y \|_{L^{\infty}}^{2(p-2)} \| v_{yy} \|_{L^{\infty}}^2 \| v_{yy} \|_{L^2}^2
            + \| v_y \|_{L^{\infty}}^{2(p-1)} \| v_{yyy} \|_{L^2}^2 ) \\
    &\le
    C  e^{-\frac{2(2\alpha - \beta +1)}{\alpha-\beta+1})s } 
    ( \| v_{yy} \|_{H^{1,0}}^2 \| v_{yy} \|_{L^2}^2
    + \| v_y \|_{H^{1,0}}^2 \| v_{yyy} \|_{L^2}^2 ) \\
    &\quad
    + C  e^{-2(p-2)s} e^{-\frac{2(2\alpha - \beta +1)}{\alpha-\beta+1})s } 
        ( \| v_y \|_{H^{1,0}}^{2(p-2)} \| v_{yy} \|_{L^{\infty}}^2 \| v_{yy} \|_{L^2}^2
            + \| v_y \|_{H^{1,0}}^{2(p-1)} \| v_{yyy} \|_{L^2}^2 ) \\
    &\le 
    C \left(
         e^{-\frac{2(2\alpha - \beta +1)}{\alpha-\beta+1})s } 
        	 ( \| f \|_{H^{2,0}}^2 + m(s)^2 )
        + e^{-2(p-2)s} e^{-\frac{2(2\alpha - \beta +1)}{\alpha-\beta+1})s } 
        	( \| f \|_{H^{2,0}}^2 + m(s)^2 )^{p-1}
    \right) \\
    &\quad
    \times ( \| f \|_{H^{3,0}}^2 + m(s)^2 )\\
    & \le
    C e^{-\frac{2\alpha -\beta +1}{\alpha - \beta +1}s } \widetilde{\mathcal{E}}(s)^2
    + e^{-\left[ 2(p-2) +\frac{2\alpha -\beta +1}{\alpha - \beta +1} \right]s }  \widetilde{\mathcal{E}}(s)^p.
\end{align} 

\subsection{Proof of a priori estimate}
Combining the energy estimates obtained in Proposition \ref{prop:energyest}
with the estimates of remainder terms given in the previous subsection, we deduce 

\begin{align}
    \qquad &\frac{d}{ds} \mathcal{E}(s) + \lambda \mathcal{E}(s)
    + \frac{1}{8} \left( \mathcal{G}(s) + m_s(s)^2 \right)\\
\label{eq:4:est:mathcalE}
    &\le
    C e^{-2\delta s} \widetilde{\mathcal{E}}(s)
    + C e^{-\frac{2\alpha -\beta +1}{\alpha - \beta +1}s } \widetilde{\mathcal{E}}(s)^2
    + C e^{-\left[ 2(p-2) + \frac{2\alpha -\beta +1}{\alpha - \beta +1} \right]s }  \widetilde{\mathcal{E}}(s)^p
\end{align} 
From Lemmas \ref{lem:Em} and \ref{lem:remainder},
we see that
\begin{align}
    \frac{d}{ds} E_{m2}(s)
    &=
    2 E_{m1}(s) +
    \left( \frac{r'}{a} - \frac{ra'}{a^2}
    \right) m(s) m_s(s) \\
    &\le
    C e^{-\delta s} m_s(s)^2
    + 
    \frac{1}{16} m_s(s)^2
    +
    C e^{-2\delta s} m(s)^2.
\end{align}
Therefore, there exists constants
$s_m \ge s_{\ast}$ and $c > 0$
such that for any $s \ge s_m$, we have
\begin{align}
    &\frac{d}{ds} \widetilde{\mathcal{E}}(s)
    + \lambda \widetilde{\mathcal{E}}(s)
    + c (\mathcal{G}(s) + m_s(s)^2) \\
\label{eq:4:est:tildeE}
    &\le
    	C_3 e^{-2\delta s} \widetilde{\mathcal{E}}(s)
    + C_3
    \left( e^{-\frac{2\alpha -\beta +1}{\alpha - \beta +1}s } \widetilde{\mathcal{E}}(s)^2
    + e^{-\left[ 2(p-2) + \frac{2\alpha -\beta +1}{\alpha - \beta +1} \right]s }  \widetilde{\mathcal{E}}(s)^p
    \right)
\end{align}
with  some $C_3>0$. Define 
\begin{align}
    \Lambda(s)
    = \exp \left( - C_3 \int_{s_m}^s e^{-2\delta \sigma} \,d\sigma \right).
\end{align}
Note that
\begin{align}
    \Lambda(s)
    = \exp \left( \frac{C_3}{2\delta}
    \left( e^{-2\delta s} - e^{-2\delta s_m} \right)
    \right)
    \sim 1
    \quad \text{and} \quad
    \Lambda (s_m) = 1.
\end{align}
Multiplying \eqref{eq:4:est:tildeE} by $\Lambda(s)$,
we deduce
\begin{align}
    &\frac{d}{ds} \left[ \Lambda(s) \widetilde{\mathcal{E}} (s) \right] 
    + \lambda \Lambda(s) \mathcal{E}(s)
    + c \Lambda(s) \left( \mathcal{G}(s) + m_s(s)^2 \right) \\
    &\le
	C_3 \Lambda(s) 
    \left( e^{-\frac{2\alpha -\beta +1}{\alpha - \beta +1}s } \widetilde{\mathcal{E}}(s)^2
    + e^{-\left[ 2(p-2) + \frac{2\alpha -\beta +1}{\alpha - \beta +1} \right]s }  \widetilde{\mathcal{E}}(s)^p
    \right).
\end{align}
Integrating the above over $[s_m, s]$, we have
\begin{align}
    \Lambda(s) \widetilde{\mathcal{E}}(s)
    &\le
    \widetilde{\mathcal{E}}(s_m)
    + C_3
    \int_{s_m}^s
    \Lambda(\sigma) 
    \left( e^{-\frac{2\alpha -\beta +1}{\alpha - \beta +1}s } \widetilde{\mathcal{E}}(\sigma)^2
    + e^{-\left[ 2(p-2) + \frac{2\alpha -\beta +1}{\alpha - \beta +1} \right]s }  \widetilde{\mathcal{E}}(\sigma)^p
    \right) \,d\sigma
\end{align}
Finally, we put
\begin{align}
    \widetilde{\mathcal{E}}_{\max} (s) = \max_{\sigma \in [s_m,s]} \widetilde{\mathcal{E}}(\sigma)
\end{align}
for $s \ge s_m$.
Then, the above estimate implies
\begin{align}
    \widetilde{\mathcal{E}}_{\max} (s)
    &\le
    C_0 \widetilde{\mathcal{E}}(s_m)
    + C_0' \left( \widetilde{\mathcal{E}}_{\max} (s)^2 + \widetilde{\mathcal{E}}_{\max} (s)^p \right)
\end{align}
with some constants $C_0, C_0' > 0$,
where we have used
$\delta > 0$
and
$p > \frac{-\beta + 1}{\alpha-\beta+1}$
(see Remark \ref{rem:11}).
Thus, we conclude the a priori estimate
\begin{align}\label{eq:4:apriori}
    \widetilde{\mathcal{E}}_{\max} (s) \le 2C_0 \widetilde{\mathcal{E}}(s_m)
\end{align}
for all $s \ge s_m$,
provided that
$\widetilde{\mathcal{E}}(s_m)$
is sufficiently small.
From the local existence result (Proposition B.2),
we see that, for sufficiently small initial data, the local solution uniquely exists over $[0,s_m]$,
and it satisfies
$\widetilde{\mathcal{E}}(s_m) \le C ( \| u_0 \|_{H^{2,1}\cap H^{3,0}} + \| u_1 \|_{H^{0,1}\cap H^{1,0}})$
(for the detail, see the proof of Proposition B.2 (vi) ).
Thus, $\widetilde{\mathcal{E}}(s_m)$ can be controlled by the norm of initial data.
This and Proposition B.2 (iii) (blow-up alternative) indicate the existence of the global solution
if the initial data $(u_0, u_1)$ is sufficiently small.

It remains to prove the asymptotic estimate.
To this end, we go back to the estimate \eqref{eq:4:est:mathcalE}.
By virtue of the a priori estimate \eqref{eq:4:apriori}, we have
\begin{align}
    \frac{d}{ds} \mathcal{E}(s) + \lambda \mathcal{E}(s)
    + \frac{1}{8} \left( \mathcal{G}(s) + m_s(s)^2 \right)
    &\le
    C e^{- \min\{ 2\delta, \frac{2\alpha -\beta +1}{\alpha - \beta +1}, 2(p-2) + \frac{2\alpha -\beta +1}{\alpha - \beta +1} \} s } 
    \widetilde{\mathcal{E}}(s_m) \\
    &=C  e^{- \min\{ \frac{2(\beta+1)}{\alpha-\beta+1}, \frac{2\alpha -\beta +1}{\alpha - \beta +1} \} s } \widetilde{\mathcal{E}}(s_m),
\end{align}
where we have also used
$\widetilde{\mathcal{E}}(s_m)$,
which can be assumed without loss of generality.
Now, recall
\begin{align}
    \lambda \in \left( 0, \min \left\{ \frac{1}{2}, \frac{2(\beta+1)}{\alpha-\beta+1}, \frac{2\alpha -\beta +1}{\alpha - \beta +1} \right\}  \right),
\end{align}
and multiply the above estimate by $e^{\lambda s}$.
Then, we obtain
\begin{align}
    \frac{d}{ds} \left[ e^{\lambda s} \mathcal{E}(s) \right]
    + \frac{1}{8} e^{\lambda s} \left( \mathcal{G}(s) + m_s(s)^2 \right)
    \le C e^{\lambda -  \min\{ \frac{2(\beta+1)}{\alpha-\beta+1}, \frac{2\alpha -\beta +1}{\alpha - \beta +1} \}  s } \widetilde{\mathcal{E}}(s_m).
\end{align}
Integrating this over $[s_m, s]$ implies
\begin{align}
    e^{\lambda s} \mathcal{E}(s)
    + \frac{1}{8} \int_{s_m}^s e^{\lambda \sigma} \left( \mathcal{G}(\sigma) + m_s(\sigma)^2 \right) \,d\sigma
    \le
    C \widetilde{\mathcal{E}}(s_m).
\end{align}
Therefore, we have
\begin{align}\label{eq:4:decay:mathcalE}
    \mathcal{E}(s) \le C e^{-\lambda s} \widetilde{\mathcal{E}}(s_m)
\end{align}
for all $s \ge s_m$.
Moreover, we deduce
\begin{align}
    \int_{s_m}^s e^{\lambda \sigma} m_s(\sigma)^2 \,d\sigma \le C \widetilde{\mathcal{E}}(s_m).
\end{align}
This shows, for any $s \ge s' \ge s_m$,
\begin{align}
    | m(s) - m (s') |
    &=
    \left| \int_{s'}^s m_s(\sigma) \,d\sigma \right| \\
    &\le
    \left( \int_{s'}^s e^{-\lambda \sigma} \,d\sigma \right)^{1/2}
        \left( \int_{s'}^{s} e^{\lambda \sigma} m_s(\sigma)^2 \,d\sigma \right)^{1/2} \\
    &\le
    \left( \frac{1}{\lambda} ( e^{-\lambda s'} - e^{-\lambda s} ) \right)^{1/2}
    C \widetilde{\mathcal{E}}_m(s_m)^{1/2} \\
    &\to 0 \quad (s',s \to \infty).
\end{align}
This means that the limit
$m^{\ast} = \lim_{s\to \infty} m(s)$
exists and satisfies
\begin{align}
    | m^{\ast} - m(s) |^2 \le C \widetilde{\mathcal{E}}(s_m) e^{-\lambda s}
\end{align}
for all $s \ge s_m$.
Consequently, by the above estimate and \eqref{eq:4:decay:mathcalE}, we have
\begin{align}
    \| v(s) - m^{\ast} \varphi \|_{L^2}^2
    &=
    \| m(s) \varphi + f(s) - m^{\ast} \varphi \|_{L^2}^2 \\
    &\le
    C \left(
        | m^{\ast} - m(s) |^2 \| \varphi \|_{L^2}^2 + \| f(s) \|_{L^2}^2 
    \right) \\
    &\le
    C e^{-\lambda s} \widetilde{\mathcal{E}}(s_m) \\
    &\le
    C e^{-\lambda s} \left( \| u_0 \|_{H^{2,1} \cap H^{3,0}}^2 + \| u_1 \|_{H^{0,1}\cap H^{1,0}} \right)^2
\end{align}
for $s \ge s_m$,
which implies
\begin{align}
    \| u(t) - m^{\ast} G(R(t)) \|_{L^2}^2
    &\le
    C (R(t)+1)^{-\frac{1}{2}-\lambda} \left( \| u_0 \|_{H^{2,1} \cap H^{3,0}}^2 + \| u_1 \|_{H^{0,1}\cap H^{1,0}} \right)^2
\end{align}
for $t \ge t_m := R^{-1}(e^s-1)$.
This completes the proof of the asymptotic estimate.

\appendix
\section{ A general lemma for the energy identity}
In this appendix, we give a proof of
Lemma \ref{lem:YoWa:en}.
Actually, we give a slightly more general version of it and prove the following lemma.
If we take $k=\frac{1}{2}$ and $c_3(s) \equiv 1$, then we have Lemma \ref{lem:YoWa:en}. 
\begin{lemma}\label{lem:YoWa:en:app}
Let
$k,l,m\in \mathbb{R}$,
$n \in \mathbb{N}\cup \{0\}$,
and let
$c_j = c_j(s) \ (j=1,2,3,4)$
be smooth functions defined on
$[0,\infty)$.
We consider a system for two functions
$f=f(s,y)$ and $g=g(s,y)$
given by
\begin{align}\label{eq:lem:YoWa}
    \left\{ \begin{alignedat}{3}
    &f_{s}-k y f_{y}-l f = g,\\
    &c_{1}(s) \left( g_{s}-k y g_{y} - m g \right)
    + c_{2}(s) g + g
    = c_{3}(s) f_{yy} - c_{4}(s) f_{yyyy}+h
    \end{alignedat} \right.
    \quad
    (s,y) \in (0,\infty) \times \mathbb{R},
\end{align}
where
$h = h(s,y)$
is a given smooth function belonging to
$C([0,\infty);H^{0,n}(\mathbb{R}))$.
We define the energies
\begin{align}
    E_1(s)
    &=
    \frac{1}{2}\int_{\mathbb{R}}
    y^{2n}
    \left(
    c_3(s) f_y^2 + c_4(s) f_{yy}^2 + c_1(s)g^2
    \right) \,dy,\\
    E_{2}(s)
    &=
    \int_{\mathbb{R}}
    y^{2n}
    \left(
    \frac{1}{2} f^{2}
    + c_{1}(s) f g
    \right) \,dy.
\end{align}
Then, we have
\begin{align}
    \frac{d}{ds} E_1(s)
    &=
    - \int_{\mathbb{R}} y^{2n} g^2 \,dy 
    + \left( - \frac{2n-1}{2}k + l \right) c_3(s)
    \int_{\mathbb{R}} y^{2n} f_y^2 \,dy
    + \left( - \frac{2n-3}{2}k + l \right)c_4(s)
    \int_{\mathbb{R}} y^{2n} f_{yy}^2 \,dy \\
    &\quad +
    \left( - \frac{2n+1}{2}k + m \right) c_1(s)
    \int_{\mathbb{R}} y^{2n} g^2 \,dy
    - c_2(s) \int_{\mathbb{R}} y^{2n} g^2 \,dy \\
    &\quad
    -2n c_3(s) \int_{\mathbb{R}} y^{2n-1} f_y g \,dy
    - 2n(2n-1) c_4(s) \int_{\mathbb{R}} y^{2n-2} f_{yy} g \,dy
    - 4n c_4(s) \int_{\mathbb{R}} y^{2n-1} f_{yy} g_y \,dy \\
    &\quad
    + \frac{c_3'(s)}{2} \int_{\mathbb{R}} y^{2n} f_y^2 \,dy 
    + \frac{c_4'(s)}{2} \int_{\mathbb{R}} y^{2n} f_{yy}^2 \,dy
    + \frac{c_1'(s)}{2} \int_{\mathbb{R}} y^{2n} g^2 \,dy
    + \int_{\mathbb{R}} y^{2n} gh \,dy
\end{align}
and
\begin{align}
    \frac{d}{ds} E_2(s)
    &=
    - c_3(s) \int_{\mathbb{R}} y^{2n} f_y^2 \,dy
    -c_4(s) \int_{\mathbb{R}} y^{2n} f_{yy}^2 \,dy 
    + \left( - \frac{2n+1}{2} k + l \right)
    \int_{\mathbb{R}} y^{2n} f^2 \,dy \\
    &\quad
    + c_1(s) \int_{\mathbb{R}} y^{2n} g^2 \,dy
    + \left( - (2n+1) k + l + m \right) c_1(s)
    \int_{\mathbb{R}} y^{2n} f g \,dy
    - c_2(s) \int_{\mathbb{R}} y^{2n} f g \,dy  \\
    &\quad
    -2n c_3(s) \int_{\mathbb{R}} y^{2n-1} f f_y \,dy 
    - 4n c_4(s) \int_{\mathbb{R}} y^{2n-1} f_y f_{yy} \,dy
    - 2n(2n-1) c_4(s) \int_{\mathbb{R}} y^{2n-2} f f_{yy} \,dy \\
    &\quad
    + c_1'(s) \int_{\mathbb{R}} y^{2n} f g \,dy
    + \int_{\mathbb{R}} y^{2n} f h \,dy.
\end{align}
\end{lemma}

\begin{proof}[Proof of Lemma \ref{lem:YoWa:en:app}]
We calculate
\begin{align}
    \frac{d}{ds} E_1(s)
    &=
    \frac{d}{ds} \left[
    \frac{1}{2}\int_{\mathbb{R}}
    y^{2n}
    \left(
    c_3(s) f_y^2 + c_4(s) f_{yy}^2 + c_1(s)g^2
    \right) \,dy \right] \\
    &=
    c_3(s) \int_{\mathbb{R}} y^{2n} f_y f_{ys} \,dy
    + \frac{c_3'(s)}{2} \int_{\mathbb{R}} y^{2n} f_y^2 \,dy \\
    &\quad +
    c_4(s) \int_{\mathbb{R}} y^{2n} f_{yy} f_{yys} \,dy
    + \frac{c_4'(s)}{2} \int_{\mathbb{R}} y^{2n} f_{yy}^2 \,dy \\
    &\quad + 
    c_1(s) \int_{\mathbb{R}} y^{2n} g g_s \,dy
    + \frac{c_1'(s)}{2} \int_{\mathbb{R}} y^{2n} g^2 \,dy.
\end{align}
Using the equation \eqref{eq:lem:YoWa},
we rewrite the above identity as
\begin{align}
    \frac{d}{ds} E_1(s)
    &=
    c_3(s) \int_{\mathbb{R}} y^{2n}
    f_y (ky f_y + l f + g )_y \,dy
    + \frac{c_3'(s)}{2} \int_{\mathbb{R}} y^{2n} f_y^2 \,dy \\
    &\quad +
    c_4(s) \int_{\mathbb{R}} y^{2n}
    f_{yy} (ky f_y + l f + g )_{yy} \,dy
    + \frac{c_4'(s)}{2} \int_{\mathbb{R}} y^{2n} f_{yy}^2 \,dy \\
    &\quad +
    c_1(s) \int_{\mathbb{R}} y^{2n}
    g (ky g_y + mg )\,dy
    -c_2(s) \int_{\mathbb{R}} y^{2n} g^2 \,dy
    - \int_{\mathbb{R}} y^{2n} g^2 \,dy \\
    &\quad +
    c_3(s) \int_{\mathbb{R}} y^{2n} g f_{yy} \,dy
    - c_4(s) \int_{\mathbb{R}} y^{2n} g f_{yyyy}\,dy
    + \int_{\mathbb{R}} y^{2n} gh \,dy \\
    &\quad +
    \frac{c_1'(s)}{2} \int_{\mathbb{R}} y^{2n} g^2 \,dy.
\end{align}
By noting the relations
\begin{align}
    y^{2n} f_y (yf_y)_y
    &=
    \left( \frac{y^{2n+1}}{2} f_y^2 \right)_y
    - \frac{2n-1}{2} y^{2n} f_y^2,\\
    y^{2n} f_{yy} (y f_y)_{yy}
    &= 
    \left( \frac{y^{2n+1}}{2} f_{yy}^2 \right)_y
    - \frac{2n-3}{2} y^{2n} f_{yy}^2,\\
    y^{2n} g (yg_y)
    &=
    \left( \frac{y^{2n+1}}{2} g^2 \right)_y
    - \frac{2n+1}{2} y^{2n} g^2,\\
    y^{2n} g f_{yy}
    &=
    \left( y^{2n} g f_y \right)_y
    - y^{2n} f_y g_y - 2n y^{2n-1} f_y g,\\
    y^{2n} g f_{yyyy}
    &=
    \left( y^{2n} g f_{yyy} \right)_y
    - \left( (y^{2n} g)_y f_{yy} \right)_y \\
    &\quad +
    \left( 2n(2n-1)y^{2n-2} g + 4n y^{2n-1} g_y + y^{2n} g_{yy} \right) f_{yy},
\end{align}
we have
\begin{align}
    \frac{d}{ds} E_1(s)
    &=
    \left( - \frac{2n-1}{2}k + l \right) c_3(s)
    \int_{\mathbb{R}} y^{2n} f_y^2 \,dy
    + c_3(s) \int_{\mathbb{R}} y^{2n} f_y g_y \,dy
    + \frac{c_3'(s)}{2} \int_{\mathbb{R}} y^{2n} f_y^2 \,dy \\
    &\quad +
    \left( - \frac{2n-3}{2}k + l \right)c_4(s)
    \int_{\mathbb{R}} y^{2n} f_{yy}^2 \,dy
    + c_4(s) \int_{\mathbb{R}} y^{2n} f_{yy} g_{yy} \,dy
    + \frac{c_4'(s)}{2} \int_{\mathbb{R}} y^{2n} f_{yy}^2 \,dy \\
    &\quad +
    \left( - \frac{2n+1}{2}k + m \right) c_1(s)
    \int_{\mathbb{R}} y^{2n} g^2 \,dy
    -c_2(s) \int_{\mathbb{R}} y^{2n} g^2 \,dy
    - \int_{\mathbb{R}} y^{2n} g^2 \,dy \\
    &\quad
    - c_3(s) \int_{\mathbb{R}} y^{2n} f_yg_y \,dy
    -2n c_3(s) \int_{\mathbb{R}} y^{2n-1} f_y g \,dy \\
    &\quad
    - 2n(2n-1) c_4(s) \int_{\mathbb{R}} y^{2n-2} f_{yy} g \,dy
    - 4n c_4(s) \int_{\mathbb{R}} y^{2n-1} f_{yy} g_y \,dy 
    - c_4(s) \int_{\mathbb{R}} y^{2n} f_{yy} g_{yy} \,dy \\
    &\quad 
    + \int_{\mathbb{R}} y^{2n} gh \,dy
    + \frac{c_1'(s)}{2} \int_{\mathbb{R}} y^{2n} g^2 \,dy.
\end{align}
Thus, we conclude
\begin{align}
    \frac{d}{ds} E_1(s)
    &=
    - \int_{\mathbb{R}} y^{2n} g^2 \,dy 
    + \left( - \frac{2n-1}{2}k + l \right) c_3(s)
    \int_{\mathbb{R}} y^{2n} f_y^2 \,dy
    + \left( - \frac{2n-3}{2}k + l \right)c_4(s)
    \int_{\mathbb{R}} y^{2n} f_{yy}^2 \,dy \\
    &\quad +
    \left( - \frac{2n+1}{2}k + m \right) c_1(s)
    \int_{\mathbb{R}} y^{2n} g^2 \,dy
    - c_2(s) \int_{\mathbb{R}} y^{2n} g^2 \,dy \\
    &\quad
    -2n c_3(s) \int_{\mathbb{R}} y^{2n-1} f_y g \,dy
    - 2n(2n-1) c_4(s) \int_{\mathbb{R}} y^{2n-2} f_{yy} g \,dy
    - 4n c_4(s) \int_{\mathbb{R}} y^{2n-1} f_{yy} g_y \,dy \\
    &\quad
    + \frac{c_3'(s)}{2} \int_{\mathbb{R}} y^{2n} f_y^2 \,dy 
    + \frac{c_4'(s)}{2} \int_{\mathbb{R}} y^{2n} f_{yy}^2 \,dy
    + \frac{c_1'(s)}{2} \int_{\mathbb{R}} y^{2n} g^2 \,dy
    + \int_{\mathbb{R}} y^{2n} gh \,dy.
\end{align}
This gives the desired identity for
$E_1(s)$.
Next, we compute
\begin{align}
    \frac{d}{ds} E_2(s)
    &=
    \frac{d}{ds}
    \left[
    \int_{\mathbb{R}}
    y^{2n}
    \left(
    \frac{1}{2} f^{2}
    + c_{1}(s) f g
    \right) \,dy
    \right] \\
    &=
    \int_{\mathbb{R}} y^{2n} f f_s \,dy
    + c_1(s) \int_{\mathbb{R}} y^{2n} f_s g \,dy
    + c_1(s) \int_{\mathbb{R}} y^{2n} f g_s \,dy
    + c_1'(s) \int_{\mathbb{R}} y^{2n} f g \,dy.
\end{align}
Using the equation \eqref{eq:lem:YoWa},
we rewrite the above identity as
\begin{align}
    \frac{d}{ds} E_2(s)
    &=
    \int_{\mathbb{R}} y^{2n} f (kyf_y+ l f + g) \,dy
    + c_1(s) \int_{\mathbb{R}} y^{2n} (ky f_y + l f + g) g \,dy \\
    &\quad
    + c_1(s) \int_{\mathbb{R}} y^{2n} f (ky g_y + mg ) \,dy
    - c_2(s) \int_{\mathbb{R}} y^{2n} f g \,dy
    - \int_{\mathbb{R}} y^{2n} f g \, dy \\
    &\quad
    + c_3(s) \int_{\mathbb{R}} y^{2n} f f_{yy} \,dy
    - c_4 (s) \int_{\mathbb{R}} y^{2n} f f_{yyyy} \,dy \\
    &\quad
    + \int_{\mathbb{R}} y^{2n} f h \,dy
    + c_1'(s) \int_{\mathbb{R}} y^{2n} f g \,dy.
\end{align}
By noting the relations
\begin{align}
    y^{2n} f (y f_y)
    &= \left( \frac{y^{2n+1}}{2} f^2 \right)_y
    - \frac{2n+1}{2} y^{2n} f^2,\\
    y^{2n} f (y g_y)
    &= \left( y^{2n+1} fg \right)_y
    - y^{2n+1} f_y g - (2n+1) y^{2n} fg,\\
    y^{2n} f f_{yy}
    &= \left( y^{2n} f f_y \right)_y
    - y^{2n} f_y^2 - 2n y^{2n-1} f f_y,\\
    y^{2n} f f_{yyyy}
    &= \left( y^{2n} f f_{yyy} \right)_y
    - \left( (y^{2n} f)_y f_{yy} \right)_y \\
    &\quad
    + (2n (2n-1) y^{2n-2} f
    + 4n y^{2n-1} f_y
    + y^{2n} f_{yy}) f_{yy},
\end{align}
we have
\begin{align}
    \frac{d}{ds} E_2(s)
    &=
    \left( - \frac{2n+1}{2} k + l \right)
    \int_{\mathbb{R}} y^{2n} f^2 \,dy
    + \int_{\mathbb{R}} y^{2n} fg \,dy \\
    &\quad
    + k c_1(s) \int_{\mathbb{R}} y^{2n+1} f_y g \,dy
    + l c_1(s) \int_{\mathbb{R}} y^{2n} f g \,dy
    + c_1(s) \int_{\mathbb{R}} y^{2n} g^2 \,dy \\
    &\quad
    -k c_1(s) \int_{\mathbb{R}} y^{2n+1} f_y g \,dy
    + \left( - (2n+1) k + m \right) c_1(s)
    \int_{\mathbb{R}} y^{2n} f g \,dy \\
    &\quad 
    - c_2(s) \int_{\mathbb{R}} y^{2n} f g \,dy
    - \int_{\mathbb{R}} y^{2n} f g \, dy \\
    &\quad
    - c_3(s) \int_{\mathbb{R}} y^{2n} f_y^2 \,dy
    -2n c_3(s) \int_{\mathbb{R}} y^{2n-1} f f_y \,dy \\
    &\quad
    - 2n(2n-1) c_4(s) \int_{\mathbb{R}} y^{2n-2} f f_{yy} \,dy
    - 4n c_4(s) \int_{\mathbb{R}} y^{2n-1} f_y f_{yy} \,dy
    -c_4(s) \int_{\mathbb{R}} y^{2n} f_{yy}^2 \,dy \\
    &\quad
    + \int_{\mathbb{R}} y^{2n} f h \,dy
    + c_1'(s) \int_{\mathbb{R}} y^{2n} f g \,dy.
\end{align}
Thus, we conclude
\begin{align}
    \frac{d}{ds} E_2(s)
    &=
    - c_3(s) \int_{\mathbb{R}} y^{2n} f_y^2 \,dy
    -c_4(s) \int_{\mathbb{R}} y^{2n} f_{yy}^2 \,dy 
    + \left( - \frac{2n+1}{2} k + l \right)
    \int_{\mathbb{R}} y^{2n} f^2 \,dy \\
    &\quad
    + c_1(s) \int_{\mathbb{R}} y^{2n} g^2 \,dy
    + \left( - (2n+1) k + l + m \right) c_1(s)
    \int_{\mathbb{R}} y^{2n} f g \,dy
    - c_2(s) \int_{\mathbb{R}} y^{2n} f g \,dy  \\
    &\quad
    -2n c_3(s) \int_{\mathbb{R}} y^{2n-1} f f_y \,dy 
    - 4n c_4(s) \int_{\mathbb{R}} y^{2n-1} f_y f_{yy} \,dy
    - 2n(2n-1) c_4(s) \int_{\mathbb{R}} y^{2n-2} f f_{yy} \,dy \\
    &\quad
    + c_1'(s) \int_{\mathbb{R}} y^{2n} f g \,dy
    + \int_{\mathbb{R}} y^{2n} f h \,dy.
\end{align}
This completes the proof.
\end{proof}


\section{Local existence}
We discuss the local existence and basic properties of
solutions to \eqref{eq:ndb}. 
Let
$X = H^{3,0}(\mathbb{R}) \times H^{1,0}(\mathbb{R})$
and
\[
    U := \begin{pmatrix}u \\ \partial_t u \end{pmatrix},\quad
    U_0 := \begin{pmatrix} u_0 \\ u_1 \end{pmatrix}.
\]
Let
$D(A) = H^{5,0}(\mathbb{R}) \times H^{3,0}(\mathbb{R})$
and define
\[
    A = \begin{pmatrix} 0 & 1 \\ -\partial_x^4 & 0\end{pmatrix},\quad
    \mathcal{T}(t) = \exp(t A).
\]
We also define
\[
    K(\sigma ; U(s) ) =
    \begin{pmatrix}
        0\\
        -b(\sigma) \partial_t u(s) + a(\sigma) \partial_x^2 u(s)
        + \partial_x N(\partial_x u(s))
    \end{pmatrix},
\]
namely,
$\sigma$ and $s$ denote the variables for the coefficients
$a(t), b(t)$ and the unknown $u$, respectively.

Now, we introduce the definition of
the strong solution and the mild solution.
\begin{definition}\label{def:sol}
Let
$I = [0,T]$ with some $T>0$
or $I = [0,\infty)$.
We say that
a function $u$ (or $U = {}^t (u, \partial_t u)$)
is a strong solution to \eqref{eq:ndb} on $I$
if
\[
    \left\{\begin{alignedat}{3}
    &U \in C(I; D(A)) \cap C^1(I;X),\\
    &\dfrac{d}{dt} U(t) = AU(t) + K(t;U(t)) \quad \text{on} \ I,\\
    &U(0) = U_0.
    \end{alignedat} \right.
\]
Also, we say that 
a function $u$ (or $U = {}^t (u, \partial_t u)$)
is a mild solution to \eqref{eq:ndb} on $I$
if
\[
    \left\{\begin{alignedat}{3}
    &U \in C(I; X),\\
    &U(t) = \mathcal{T}(t) U_0 + \int_0^t \mathcal{T}(t-s) K(s; U(s)) \,ds \quad \text{in} \ C(I; X).\\
    \end{alignedat} \right.
\]
\end{definition}
\begin{proposition}\label{prop:locsol}
\textup{(i) (Local existence)}
For any $U_0 \in X$, there exists $T >0$ such that
there exists a mild solution to \eqref{eq:ndb} on
$[0,T]$.

\noindent
\textup{(ii) (Uniqueness)}
Let $T > 0$.
If $U$ and $V$ are mild solutions in 
$C([0,T] ; X)$
with the same initial condition
$U(0)=V(0) = U_0$,
then $U = V$.

\noindent
\textup{(iii) (Blow-up alternative)}
Let
$T_{\max} = T_{\max}(U_0)$
be
\[
    T_{\max} = \sup \{ T \in (0,\infty] ; \,
    \exists U \in C([0,T] ; X) : \text{a mild solution to \eqref{eq:ndb}}
    \}.
\]
If $T_{\max} < \infty$, then
$\lim_{t\to T_{\max-0}} \| U(t) \|_X = \infty$.

\noindent
\textup{(iv) (Continuous dependence on the initial data)}
Let
$U_0 \in X$
and
$\{ U_0^{(j)} \}_{j=1}^{\infty}$
a sequence in $X$ satisfying
$\lim_{j\to \infty} \| U_0^{(j)} - U_0 \|_X = 0$.
Let
$U$ and $U^{(j)}$ be the corresponding mild solutions
to the initial data $U_0$ and $U_0^{(j)}$, respectively.
Then, for any fixed
$T \in (0, T_{\max}(U_0))$,
we have
$T_{\max}(U_0^{(j)}) > T$
for sufficiently large $j$ and
\[
    \lim_{j\to \infty}
    \sup_{t \in [0,T]} \| U^{(j)}(t) - U(t) \|_X = 0.
\]

\noindent
\textup{(v) (Regularity)}
Let $T > 0$.
If $U_0 \in D(A)$, then the mild solution in
\textup{(i)} on $[0,T]$
becomes a strong solution on $[0,T]$.

\noindent
\textup{(vi) (Small data almost global existence)}
For any $T>0$, there exists
$\varepsilon_0 > 0$ such that if
$\| U_0 \|_X < \varepsilon_0$,
then the corresponding mild solution $U$ can be extended to $[0,T]$.

\noindent
\textup{(vii) (Boundedness of weighted norm)}
Let $T>0$ and
$Y:= H^{2,1}(\mathbb{R}) \times H^{0,1}(\mathbb{R})$.
If $U_0 \in X \cap Y$,
then the corresponding mild solution $U$ on $[0,T]$ belongs to
$C([0,T]; X \cap Y)$.
\end{proposition}
\begin{proof}
Let
$T_0 > 0$
be fixed.
Then, $a(t), b(t)$ are positive,
and they and their first derivatives are
bounded by some constant $C_{T_0}>0$
on $[0,T_0]$.
Let $T \in (0,T_0]$.
Then, for any $U = {}^t (u,v) \in X$ and $t \in [0,T]$,
we have
\begin{align}
    \| K(t; U) \|_{X}
    &=
    \left\| -b(t) \partial_t u + a(t) \partial_x^2 u + \partial_x N (\partial_x u) \right\|_{H^1} \\
    &\le
    C_{T_0} \left( \| \partial_t u \|_{H^1} + \| \partial_x^2 u \|_{H^1} \right) \\
    &\quad +
    C \left( \| \partial_x u \|_{W^{1,\infty}} + \| \partial_x u \|_{W^{1,\infty}}^{p-1} \right)
    \| \partial_x^2 u \|_{H^1}
    < \infty,
\end{align}
that is,
$K(t; \cdot) : X \to X$.
Moreover, for $M> 0$ and
$U = {}^t(u,v), W = {}^t(w,z) \in
B_M = \{ U \in X ; \| U \|_X \le M \}$,
we calculate
\begin{align}
    \| K (t; U) - K(t;W) \|_X
    &\le
    C_{T_0} \left( \| v - z \|_{H^1} + \| u - w \|_{H^3} \right) \\
    &\quad 
    + C \left( \| u \|_{W^{2,\infty}}
        + \| w \|_{W^{2,\infty}} \right) \|  u -  w \|_{H^3} \\
    &\quad
    + C \left( | u \|_{W^{2,\infty}}
        + \| w \|_{W^{2,\infty}} \right)^{p-1}
        \| u - w \|_{H^3} \\
    &\le
    C_{T_0,M} \| U - W \|_X.
\end{align}
Therefore,
$K(t; \cdot)$
is locally Lipschitz continuous in $X$.
Therefore, from the proofs of \cite[Lemmas 4.3.2, Proposition 4.3.3]{CaHa},
there exist $T > 0$ and a unique mild solution $u$
on $I = [0,T]$.
Also, \cite[Theorem 4.3.4]{CaHa} shows the property (iii).
Moreover, by \cite[Proposition 4.3.7]{CaHa},
the continuous dependence on the initial data.
This proves (i)--(iv).

Next, we prove (iv) along with the argument of
\cite[Lemma 4.3.9]{CaHa}.
Take $U_0 \in D(A)$ and $T \in (0, T_{\max})$.
Let
$h>0$, $t \in [0,T-h]$, and
$M : = \sup_{s\in [0,T]} \| U(s) \|_X$.
Consider
\begin{align}
    U(t+h) - U(t)
    &=
    \mathcal{T}(h)\mathcal{T}(t) U_0 - \mathcal{T}(t) U_0 \\
    &\quad
    + \int_0^t \mathcal{T}(s)
    \left\{ K(t+h-s; U(t+h-s)) - K(t-s;U(t-s)) \right)\} \,ds\\
    &\quad
    + \int_0^h \mathcal{T}(t+s) K(h-s ; U(h-s)) \,ds \\
    &=: J_1 + J_2 + J_3.
\end{align}
For $J_1, J_2$, we estimate
\begin{align}
    \| J_1 \|_X
    &\le
    \| \mathcal{T}(h) U_0 - U_0 \|_X
    = \left\| \int_0^h \mathcal{T}(s) A U_0 \,ds \right\|_X
    \le h \| A U_0 \|_X,\\
    \| J_3 \|_X
    &\le
    h \sup_{s \in [0,T]} \| K(s;U(s)) \|_X.
\end{align}
For $J_2$, using the
Lipschitz continuity of $a, b : [0,T] \to \mathbb{R}$
and $K(s : \cdot ) : X \to X$,
we can show
\begin{align}
    \| J_2 \|_X
    &\le
    \int_0^t \| K(t+h-s ; U(t+h-s)) - K(t+h-s; U(t-s)) \|_X \,ds \\
    &\quad
    + \int_0^t \| K(t+h-s;U(t-s))-K(t-s;U(t-s)) \|_X \,ds \\
    &\le
    C_{T_0,M} h + C_{T_0,M} \int_0^t \| U(s+h) - U(s) \|_X \,ds.
\end{align}
Then, the Gronwall inequality implies
\[
    \| U(t+h)- U(t) \|_X
    \le C_{T_0,M} h,
\]
that is,
$U : [0,T] \to X$
is Lipschitz continuous.
This further leads to
\begin{align}
    \| K(t;U(t)) - K(s;U(s)) \|_X
    &\le
    \| K(t;U(t)) - K(s;U(t)) \|_X
    + \| K(s;U(t)) - K(s;U(s)) \|_X \\
    &\le C_{T_0,M} |t-s|,
\end{align}
i.e.,
$K(\cdot ; U(\cdot) ) : [0,T] \to X$
is Lipschitz continuous,
and hence,
$K(\cdot ; U(\cdot)) \in W^{1,1}((0,T); X)$.
This enables us to apply \cite[Lemma 4.16]{CaHa}
and $u$ becomes a strong solution.
This proves (v).

Next, we prove (vi).
Let $T>0$ be arbitrary fixed,
$I := [0,T]$,
and
\[
    C_{T,a,b} := \int_0^T ( |a(s)| + |b(s)| ) \,ds.
\]
Let
$\varepsilon > 0$
be sufficiently small so that
$2(1+C_{T,a,b})\varepsilon <1$
and let
$\mathcal{B}_{\varepsilon} 
= \{ U \in C([0,T] ; X) ;\, \sup_{t \in [0,T]} \| U(t) \|_X \le 2(1+C_{T,a,b})\varepsilon \}$.
Define a map
$\Phi : C(I:X) \to C(I;X)$
by
\[
    \Phi[U] (t) :=
    \mathcal{T}(t) U_0 + \int_0^t \mathcal{T}(t-s) K(s;U(s)) \,ds.
\]
Then, for $U_0$ satisfying
$\| U_0 \|_X \le \varepsilon$
and
$U = {}^t (u,v) \in \mathcal{B}_{\varepsilon}$,
we see that
\begin{align}
    \| \Phi[U] (t) \|_X
    &\le
    \| \mathcal{T}(t) U_0 \|_X
    + \int_0^t \| \mathcal{T}(t-s) K(s;U(s)) \|_X \,ds \\
    &\le
    \| U_0 \|_X
    + \int_0^t \left( |b(s)| \| v(s) \|_{H^1} + |a(s)| \| \partial_x^2 u(s) \|_{H^1} \right) \,ds \\
    &\quad + \int_0^t \| \partial_x N (\partial_x u(s)) \|_{H^1} \,ds \\
    &\le
    (1+C_{T,a,b}) \varepsilon
    + T C_N (2(1+C_{T,a,b}) \varepsilon )^2,
\end{align}
where
$C_N > 0$ is a constant depending only on the nonlinearity $N$.
Similarly, we have, for $U, V \in \mathcal{B}_{\varepsilon}$,
\begin{align}
    \| \Phi[U](t) - \Phi[V] (t) \|_X
    &\le
    \int_0^t \| K(s; U(s)) - K(s ; V(s)) \|_X \,ds \\
    &\le T \tilde{C}_N (2(1+C_{T,a,b})\varepsilon) 
    \sup_{s \in [0,T]} \| U(s) - V(s) \|_X, 
\end{align}
$\tilde{C}_N > 0$ is a constant depending only on the nonlinearity $N$.
Therefore, taking $\varepsilon$ further small so that
\[
    T C_N (2(1+C_{T,a,b}) \varepsilon ) \le 1,\quad
    T \tilde{C}_N (2(1+C_{T,a,b})\varepsilon) \le \frac{1}{2},
\]
we see that
$\Phi$ is a contraction mapping on $\mathcal{B}_{\varepsilon}$.
This and the uniqueness of mild solution imply that
the mild solution obtained in (i) can be extended to $[0,T]$.

Finally, we prove (vii).
Let $T> 0$, $I = [0,T]$, $U_0 \in Y$, and
let $U$ be the corresponding mild solution on $[0,T]$
to the initial data $U_0$.
We put
$M := \sup_{t\in I} \| U(t) \|_X$.
In order to justify the following energy method,
we take a sequence
$\{ U_0^{(j)} \}_{j=1}^{\infty}$
from $[C_0^{\infty}(\mathbb{R})]^2$
such that
$\lim_{j\to \infty} U_0^{(j)} = U_0$
in $X \cap Y$.
Then, the corresponding strong solution 
$U^{(j)} \in C(I ; D(A))\cap C^1(I; X)$
to the data $U_0^{(j)}$
satisfies
$\lim_{j\to \infty} U^{(j)} = U$
in
$C(I ; X)$
by the continuous dependence on the initial data.
In particular, taking sufficiently large $j$,
we may suppose that
$\sup_{j\in \mathbb{N}, t \in I} \| U^{(j)}(t) \|_X \le 2 M$.

Let
\begin{align*}
    &\chi \in C_0^{\infty} (\mathbb{R}),\quad
    0\le \chi \le 1, \quad
    \chi(x) = \begin{cases} 1 &(|x|\le 1),\\ 0 &(|x| \ge 2), \end{cases}\\
    &\chi_n(x) := \chi \left( \frac{x}{n} \right) \quad
    (n \in \mathbb{N}).
\end{align*}
By
$\mathrm{supp}\, \chi_n \subset [-2n, 2n]$,
we easily see that
\begin{align}
    | \partial_x (x^2 \chi_n(x)^2 ) |
    &=
    \left| 2 x \chi_n(x)^2 + 2 \frac{x^2}{n} \chi' \left( \frac{x}{n} \right) \chi_n(x) \right|
    \le C |x| \chi_n(x),\\
    | \partial_x^2 (x^2 \chi_n(x)^2 ) |
    &=
    \left|
    2 \chi_n(x)^2 + 4 \frac{x}{n} \chi' \left( \frac{x}{n} \right) \chi_n(x)
    + 2 \frac{x^2}{n^2} \left( \left( \chi'\left(\frac{x}{n}\right)\right)^2 + \chi''\left( \frac{x}{n} \right) \chi_n(x) \right)
    \right| \\
    &\le C 
\end{align}
with some constant $C>0$.
Denote
$U = {}^t (u, \partial_t u)$,
$U^{(j)} = {}^t (u^{(j)}, \partial_t u^{(j)})$,
and consider
\begin{align*}
    E_n (t;u)
    &:= \int_{\mathbb{R}}
    x^2 \chi_n(x)^2
    \left(
        |\partial_t u(t,x)|^2
        + a(t) |\partial_x u(x)|^2
        + |\partial_x^2 u(t,x)|^2
        + | u(t,x)|^2 
    \right) \,dx,\\
    E (t;u)
    &:= \int_{\mathbb{R}}
    x^2
    \left(
        |\partial_t u(t,x)|^2
        + a(t) |\partial_x u(x)|^2
        + |\partial_x^2 u(t,x)|^2
        + | u(t,x)|^2 
    \right) \,dx.
\end{align*}
Note that $E_n(t;u^{(j)})$ is finite thanks to
$\chi_n$.
Differentiating it, we have
\begin{align}
    \frac{d}{dt} E_n (t; u^{(j)})
    &=
    2 \int_{\mathbb{R}} x^2 \chi_n(x)^2
    \left( 
    \partial_t u^{(j)} \partial_t^2 u^{(j)}
    + a(t) \partial_x u^{(j)} \partial_t \partial_x u^{(j)}
    + \partial_x^2 u^{(j)} \partial_t \partial_x^2 u^{(j)}
    \right) \,dx \\
    &\quad
    + 2 \int_{\mathbb{R}} x^2 \chi_n(x)^2 u^{(j)} \partial_t u^{(j)} \,dx
    + \int_{\mathbb{R}} x^2 \chi_n(x)^2 a'(t) |\partial_x u^{(j)} |^2 \,dx.
\end{align}
By the integration by parts and using the equation \eqref{eq:ndb}, the right-hand side can be written as
\begin{align}
    &2 \int_{\mathbb{R}} x^2 \chi_n(x)^2 \partial_t u^{(j)}
    \left( -b(t) \partial_t u^{(j)} + \partial_x N(\partial_x u^{(j)} ) \right) \,dx \\
    &\quad
    -2 \int_{\mathbb{R}} \partial_x (x^2 \chi_n(x)^2 ) a(t) \partial_x u^{(j)} \partial_t u^{(j)} \,dx \\
    &\quad
    + 4 \int_{\mathbb{R}}  \partial_x (x^2 \chi_n(x)^2 ) a(t) \partial_x^3 u^{(j)} \partial_t u^{(j)} \,dx
    +2 \int_{\mathbb{R}}  \partial_x^2 (x^2 \chi_n(x)^2 ) a(t) \partial_x^2 u^{(j)} \partial_t u^{(j)} \,dx \\
    &\quad
    + 2 \int_{\mathbb{R}} x^2 \chi_n(x)^2 u^{(j)} \partial_t u^{(j)} \,dx
    + \int_{\mathbb{R}} x^2 \chi_n(x)^2 a'(t) |\partial_x u^{(j)} |^2 \,dx.
\end{align}
The above quantity can be further estimated by
\[
    C (2M)^2
    + C_{T,a,b,M} E_n (t; u^{(j)})
\]
with some constants $C, C_{T,a,b,M} > 0$.
Hence, the Gronwall inequality implies
\[
    E_n (t; u^{(j)} ) \le \tilde{C}_{T,a,b,M},
\]
where the constant $\tilde{C}_{T,a,b,M}$
is independent of $n$ and $j$.
Letting $j \to \infty$ first and using
the continuous dependence on the initial data, we have
\[
    \int_{\mathbb{R}}
    x^2 \chi_n(x)^2
    \left(
        |\partial_t u (t,x)|^2
        + a(t) |\partial_x u (x)|^2
        + |\partial_x^2 u (t,x)|^2
        + | u (t,x)|^2 
    \right) \,dx
    \le \tilde{C}_{T,a,b,M}.
\]
Then, letting $n \to \infty$, we conclude
\[
    \int_{\mathbb{R}}
    x^2
    \left(
        |\partial_t u (t,x)|^2
        + a(t) |\partial_x u (x)|^2
        + |\partial_x^2 u (t,x)|^2
        + | u (t,x)|^2 
    \right) \,dx
    \le \tilde{C}_{T,a,b,M},
\]
which shows
$U(t) \in Y$
for any $t \in [0,T]$.
The continuity of $\| U(t) \|_Y$
in $t$
follows from the estimate
\begin{align}
    | E_n(t; u^{(j)}) - E_n (s; u^{(j)}) |
    &\le
    \int_s^t \left| \frac{d}{d\sigma} E_n(\sigma ; u^{(j)}) \right| \,d\sigma 
    \le
    C_{T,a,b,M} (t-s)
\end{align}
for $s < t$
and taking the limits $j\to \infty$ and $n\to \infty$.

\end{proof}

\section*{Acknowledgements}
This work was supported by JSPS KAKENHI Grant Numbers
JP18H01132,
JP20K14346.

\end{document}

%% file: zu2.tex
\unitlength 0.1in
\begin{picture}( 31.9000, 23.3500)(  0.0000,-35.2500)
\put(30.6000,-24.6000){\makebox(0,0)[lt]{$\al$}}%
\put(17.5000,-13.6000){\makebox(0,0)[rb]{$\bet$}}%
\put(18.4000,-24.4000){\makebox(0,0)[lt]{$0$}}%
\put(17.6000,-19.7000){\makebox(0,0)[rb]{$1$}}%
\put(18.4000,-28.6000){\makebox(0,0)[lt]{$-1$}}%
\put(16.7000,-23.8000){\makebox(0,0)[rb]{$-\frac{1}{2}$}}%
\put(6.6000,-28.4000){\makebox(0,0)[lt]{$\bet=-1$}}%
\put(21.2000,-17.2000){\makebox(0,0)[lt]{$\beta = \al +1$}}%
\put(15.3000,-26.6000){\makebox(0,0)[rb]{$\bet = 2 \al +1$}}%
\put(24.4000,-22.1000){\makebox(0,0){$\Omega_1$}}%
\put(9.9000,-23.3000){\makebox(0,0)[rb]{$\Omega_2$}}%
\put(20.0000,-31.8000){\makebox(0,0)[lb]{$\Omega_3$}}%
\put(9.3000,-31.9000){\makebox(0,0){$\Omega_4$}}%
\put(12.5000,-15.8000){\makebox(0,0){$\Omega_5$}}%
\put(18.2000,-36.1000){\makebox(0,0){Figure 1.}}%
%
\special{pn 20}%
\special{pa 1810 2000}%
\special{pa 1540 2810}%
\special{fp}%
%
\special{pn 20}%
\special{pa 1540 2800}%
\special{pa 1540 3390}%
\special{fp}%
%
\special{pn 20}%
\special{pa 610 2810}%
\special{pa 3190 2810}%
\special{fp}%
%
\special{pn 20}%
\special{pa 1810 2000}%
\special{pa 600 2000}%
\special{fp}%
%
\special{pn 8}%
\special{pa 600 2400}%
\special{pa 3170 2400}%
\special{fp}%
\special{sh 1}%
\special{pa 3170 2400}%
\special{pa 3104 2380}%
\special{pa 3118 2400}%
\special{pa 3104 2420}%
\special{pa 3170 2400}%
\special{fp}%
%
\special{pn 8}%
\special{pa 1810 3400}%
\special{pa 1810 1200}%
\special{fp}%
\special{sh 1}%
\special{pa 1810 1200}%
\special{pa 1790 1268}%
\special{pa 1810 1254}%
\special{pa 1830 1268}%
\special{pa 1810 1200}%
\special{fp}%
%
\special{pn 20}%
\special{pa 1810 2000}%
\special{pa 2570 1230}%
\special{fp}%
\end{picture}%